\documentclass[10pt,reqno]{amsart}
\usepackage{amsmath,amssymb,amsthm,fullpage, graphicx,chngcntr,float,xcolor,verbatim,cases,comment,dsfont}
\usepackage{subcaption}
\usepackage{hyperref} 
\usepackage{pgfplots}
\pgfplotsset{width=10cm,compat=1.9}

\newtheorem{theorem}{Theorem}[section]
\newtheorem{lemma}[theorem]{Lemma}

\newtheorem{proposition}[theorem]{Proposition}

\theoremstyle{definition}

\newtheorem{remark}[theorem]{Remark}
\newtheorem{example}[theorem]{Example}

\newcommand{\N}{\mathbb{N}}

\newcommand{\R}{\mathbb{R}}

\newcommand{\Fh}{\mathcal{F}_h}
\newcommand{\cci}{C_c^\infty}
\newcommand{\grad}{\nabla}
\newcommand{\xib}{\bar{\xi}}
\newcommand{\xb}{\bar{x}}
\newcommand{\yb}{\bar{y}}

\newcommand{\supp}{\operatorname{supp}}
\newcommand{\bm}{\boldsymbol{m}}

\counterwithin{figure}{section}

\title{$L^p$ estimates for joint quasimodes of two pseudodifferential operators whose characteristic sets have $k$-th Order Contact}
\author{Madelyne M. Brown and Melissa Tacy}
\address{Department of Mathematics, University of Auckland, Auckland, New Zealand}
\email{madelyne.brown@auckland.ac.nz}

\address{Department of Mathematics, University of Auckland, Auckland, New Zealand}
\email{melissa.tacy@auckland.ac.nz}
\date{April 2025}

\begin{document}
\begin{abstract}
   On a smooth, compact, $n$-dimensional Riemannian manifold, we consider functions $u_h$ that are joint quasimodes of two semiclassical pseudodifferential operators $p_1(x,hD)$ and $p_2(x,hD)$. We develop $L^p$ estimates for $u_h$ when the characteristic sets of $p_1$ and $p_2$ meet with $k$-th order contact. This paper is the natural extension of the two-dimensional result of \cite{TacyContact} to $n$ dimensions.
\end{abstract}

\maketitle

\section{Introduction and Main Results}\label{sec:intro}
Let $(M,g)$ be a smooth, compact, $n$-dimensional Riemannian manifold without boundary and let $u_h$ denote the Laplace eigenfunctions on $M$, which satisfy
\begin{equation}\label{eqn:efcnh}
    (-h^2\Delta_g- 1) u_{h}=0. 
\end{equation}
Understanding the growth and concentration of $\{u_{h}\}$ is an important and well-studied problem in mathematical physics, as the eigenfunctions can be used to describe many physical phenomena. For example,  $|u_{h}(x)|^2$ gives the probability density function of finding a free quantum particle of energy $h^{-1}$ at $x\in M$. Unfortunately, explicitly computing the eigenfunctions can only be done when $M$ is highly symmetric (e.g.\ $M$ is a sphere or torus). 

In this note, we study the behaviour of Laplace eigenfunctions by comparing the $L^p$ norm of $u_{h}$ to its $L^2$ norm. Particularly, we study the $L^p$ norms of functions $u_h$ that approximately solve (\ref{eqn:efcnh}) and additionally approximately satisfy a second pseudodifferential equation. 

\subsection{A brief overview of $L^p$ estimates}\label{sec:hist}
Since the late 1900s, $L^p$ estimates have been used to discern the concentration of eigenfunctions. In 1988, Sogge \cite{SoggeLp} showed
\begin{equation}\label{eqn:sogge}
\|u_{h}\|_{L^p(M)}\lesssim h^{-\delta(n,p)}\|u_{h}\|_{L^2(M)} \quad \text{where} \quad \delta(n,p)=\begin{cases} \frac{n-1}{4} -\frac{n-1}{2p}& 2 \leq p \leq \frac{2(n+1)}{n-1} \\ \frac{n-1}{2}-\frac{n}{p} & \frac{2(n+1)}{n-1} \leq p \leq \infty \end{cases}.
\end{equation}
The bound in (\ref{eqn:sogge}) is sharp, as there are examples on the sphere, namely the highest weight spherical harmonics and the zonal harmonics, which saturate the bound. We note that the estimate in (\ref{eqn:sogge}) is not simply due to interpolating between the $L^2$ and $L^\infty$ estimates. Indeed, interpolating between $L^2$ and the $L^\infty$ estimate $|u_{h}|\lesssim h^{-\frac{n-1}{2}}$ (which is a consequence of the Pointwise Weyl Law of Avakumovi\'{c}, Levitan, and H\"{o}rmander \cite{Avak,Lev,Hor}) gives
\[
\|u_{h}\|_{L^p(M)}\lesssim h^{-\frac{n-1}{2}+\frac{n-1}{p}} \|u_{h}\|_{L^2(M)},
\]
which is not sharp for $p\in(2,\infty)$. In \cite{KTZ}, the authors generalise (\ref{eqn:sogge}) to solutions of Schr\"{o}dinger operators and non-degenerate semiclassical operators. 

$L^p$ estimates have also been studied for restricted eigenfunctions. For example, in \cite{BGT} and \cite{Hu}, the authors show for $H$ a submanifold of dimension $d$ that
\begin{equation}\label{eqn:lpsub}
\|u_h\|_{L^p(H)}\lesssim h^{-\delta(n,d,p)}\|u_h\|_{L^2(M)} \quad \text{where}, \quad  \delta(n,d,p)=\begin{cases} \frac{n-1}{2}-\frac{d}{p} & \text{if} \begin{cases} 1 \leq d \leq n-3 \\ d=n-2 \text{ and } 2<p\leq \infty \\ d=n-1 \text{ and } p_0 < p \leq \infty \end{cases}  \\ \frac{n-1}{4}-\frac{d-1}{2p} & \text{if } d=n-1 \text{ and } 2\leq p < p_0 \end{cases},
\end{equation}
where $p_0=\frac{2n}{n-1}$. In \cite{TacySub}, the second-named author extends (\ref{eqn:lpsub}) to solutions of semiclassical pseudodifferential operators. In the case where $H$ is a smooth, curved hypersurface ($d=n-1$), Hassell and Tacy \cite{HassTacy} obtain an improvement upon (\ref{eqn:lpsub}).

Following Sogge's initial $L^{p}$ estimates \cite{SoggeLp}, there has been much interest around the question, ``under what conditions can \eqref{eqn:sogge} be improved?''. In \cite{HasselTacylog} Hassell and Tacy show logarithmic improvements for high $p$ ($p>\frac{2(n+1)}{n-1}$) on manifolds with non-positive curvature. In \cite{CanzaniGalkowski} Canzani and Galkowski use their geodesic beam technique to characterise some dynamical conditions that ensure improvements over \eqref{eqn:sogge}.   In the low $p$ case ($p<\frac{2(n+1)}{n-1}$), Sogge and Zelditch  \cite{SoggeZelditch2D} then Blair and Sogge \cite{BlairSogge1,BlairSogge2,BlairSogge3} obtained logarithmic improvements under the non-positive curvature assumption. Addressing the critical $p$ value, $p=\frac{2(n+1)}{n-1}$ is, from a technical viewpoint, more difficult. In a series of papers; Sogge \cite{Soggecritp}, Blair and Sogge \cite{BlairSoggecritp}, and Blair, Huang and Sogge \cite{BlairHuangSoggecritp} obtained logarithmic improvements at the critical $p$ with the final set of estimates appearing in \cite{Soggesurvey} (along with a detailed discussion about the curvature conditions and the importance and difficulty of the critical value of $p$). 

In this note, we consider improvement to \eqref{eqn:sogge} when $u_h$ is an eigenfunction that additionally solves a secondary equation. That is 
\begin{equation}\label{eqn:2ops}
0=(-h^2\Delta-1)u_h=:p_1(x,hD)u_h\quad \text{and} \quad 0=p_2(x,hD)u_h,
\end{equation}
where $p_2$ is some semiclassical pseudodifferential operator. In \cite{TacyTrans}, the author obtains $L^p$ estimates for eigenfunctions that additionally satisfy $r-1$ semiclassical equations, $p_2(x,hD)u_h=\dots p_r(x,hD)u_h=0$, under the assumption that the characteristic sets $\{p_j(x,\xi)=0\}$ meet transversely. Particularly, in this case, Tacy shows
\begin{equation}\label{eqn:lpr}
\|u_h\|_{L^p(M)}\lesssim h^{-\delta(n,p,r)}\|u\|_{L^2(M)} \quad \text{where} \quad \delta(n,p,r)=\begin{cases} \frac{n-r}{4}-\frac{n-r}{2p} & 2\leq p \leq \frac{2(n-r+2)}{n-r} \\ \frac{n-r}{2}-\frac{n-r+1}{p} &\frac{2(n-r+2)}{n-r} \leq p \leq \infty \end{cases}. 
\end{equation}
This article is inspired by the work \cite{TacyContact} of the second-named author, who considers eigenfunctions satisfying one additional semiclassical equation (\ref{eqn:2ops}) on a compact Riemannian surface $(n=2)$, where the characteristic curves,
\[
\{\xi\in T^*_{x_0}(M):|\xi_{g(x_0)}^2-1=0\} \quad \text{and} \quad \{\xi\in T^*_{x_0}(M):p_2(x_0,\xi)=0\},
\]
meet at a single point $\xi=\xi_0$ and have $k$-th order contact for $k\geq 1$. Here, Tacy uses the natural notion for the contact of curves determined by the order in which the derivatives of these characteristic curves agree. Since $k\geq 1$, the characteristic sets are, at a minimum, tangential at the point of contact. This note aims to extend Tacy's work into $n$-dimensions. Instead of using the contact of characteristic curves, we must consider the structure of the contact between the $(n-1)$-dimensional characteristic sets. 

Many of the above results also hold for a more general class of functions. That is, we do not need to require that $u_h$ solves (\ref{eqn:efcnh}) exactly, instead it may be an ``approximate solution" or \textit{quasimode of order} $h$ satisfying
\[
\|(-h^2\Delta-1)u_h\|_{L^2}\lesssim h \|u_h\|_{L^2}.
\]
Furthermore, we may often replace $(-h^2\Delta-1)$ with more general ``Laplace-like" operators $P_1:=p_1(x,hD)$. For example, we could take $P_1$ that has a characteristic set,
\[
\{\xi:p_1(x,\xi)=0\},
\]
which is curved (i.e.\ has a non-degenerate second fundamental form).

\subsection{Statement of results}\label{sec:sor} Let $M$ be a smooth, compact, boundaryless, $n$-dimensional Riemannian manifold. Let $P_1:=p_1(x,hD)$ and $P_2:=p_2(x,hD)$ be two semiclassical pseudodifferential operators mapping $L^2(M)\to L^2(M)$. Moreover, suppose $P_1$ has curved characteristic set. Throughout this article, we use the left quantisation, that is
\[
p(x,hD)u(x)=\frac{1}{(2\pi h)^n}\int e^{\frac{i}{h}\langle x-y,\xi\rangle}p(x,\xi)u(y) dy \, d\xi.
\]
Let $u_h$ be an approximate solution of order $h$ of both $P_1$ and $P_2$ so
\[
p_1(x,hD)u_h=O_{L^2}(h\|u_h\|_{L^2}) \quad \text{and} \quad p_2(x,hD)u_h=O_{L^2}(h\|u_h\|_{L^2}).
\]
Furthermore, we assume $u_h$ is a \textit{strong joint quasimode of order $h$} for $P_1$ and $P_2$. That is, for all $M_1,M_2\in\N$ we have
\[
\|P_1^{M_1} P_2^{M_2} u_h\|_{L^2(M)}\lesssim h^{M_1+M_2}\|u_h\|_{L^2(M)}.
\]
Moreover, we also assume $u_h$ is \textit{compactly microlocalised}, so there exists $\chi\in \cci(T^* M)$ such that
\[(1-\chi(x,hD))u_h=O(h^\infty).\]

In this article, we consider the case where the characteristic sets,
\[
\{\xi\in T^*_{x_0}(M):p_1(x_0,\xi)=0\} \quad \text{and} \quad \{\xi\in T^*_{x_0}(M):p_2(x_0,\xi)=0\},
\]
 have higher-order contact, so, they are at a minimum tangential at the point of contact. In contrast to the $n=2$ case studied in \cite{TacyContact} where the characteristic sets are curves, the characteristic sets in the $n$-dimensional case are smooth hypersurfaces. The standard way to define the order of contact between curves is via the derivative. In higher dimensions, no canonical definition of the order of contact exists. In this note, we will define the contact of the characteristic sets by considering the order of contact between all pairs of geodesics on the hypersurfaces passing through the contact point. However, it is not necessarily the case that all pairs of geodesics have the same order of contact. For example, the surfaces in $\R^3$ meeting at $0$ defined by
\[
S_1=\{(x,y,z):z=0\} \quad \text{and} \quad S_2=\{(x,y,z):z=x^2+y^6\}
\]
have first-order contact along all lines $\{y=mx\}$ for any $m\in \R$. However, along the line $\{x=0\}$, the contact is fifth order. In this article, we assume the order of contact is the same in each direction.

\begin{theorem}\label{thrm:linfcontact} Suppose $u_h$ is a compactly microlocalised, strong joint quasimode of order $h$ for $P_1=p_1(x,hD)$ and $P_2=p_2(x,hD)$. Let $k\geq 1$ and odd. Furthermore assume $p_1(x,\xi)$ and $p_2(x,\xi)$ satisfy
\begin{enumerate}
    \item For each $x_0\in M$ and $j=1,2$ the set $\{\xi:p_j(x_0,\xi)=0 \}=:H_{x_0}^j$ is a smooth hypersurface in $T_{x_0}^* M$.
    \item For each $x_0$, the hypersurfaces $H_{x_0}^1,H_{x_0}^2$ meet at a single point $\xi_0$. Moreover, assume that $H_{x_0}^1,H_{x_0}^2$ are tangential at $\xi_0$.
    \item Let $\gamma^1_{\xi_0,v}(t)\in H_{x_0}^1$ and $\gamma^2_{\xi_0,v}(t)\in H_{x_0}^2$ denote a pair of geodesics in $H_{x_0}^1,H_{x_0}^2$ respectively passing through $\xi_0$ at direction $v$. Hence, $\gamma^j_{\xi_0,v}(0)=\xi_0$ and $\dot{\gamma}^j_{\xi_0,v}(0)=v$. Moreover, suppose  $\gamma^1_{\xi_0,v}(t)$ and $\gamma^2_{\xi_0,v}(t)$ meet with $k$-th order contact for all directions $v$, so,
    \[
    \frac{d^j}{dt^j}\gamma^1_{\xi_0,v}(0)=\frac{d^j}{dt^j}\gamma^2_{\xi_0,v}(0) \qquad \text{for } 0 \leq j \leq k, \qquad \text{and}  \quad  \frac{d^{k+1}}{dt^{k+1}}\gamma^1_{\xi_0,v}(0) \not=\frac{d^{k+1}}{dt^{k+1}}\gamma^2_{\xi_0,v}(0).
    \]
    \item For all $x_0$, $\{\xi:p_1(x_0,\xi)\}$ has non-degenerate second fundamental form.
\end{enumerate}
Then
\begin{equation}\label{eqn:linf}
\|u_h\|_{L^p(M)}\lesssim h^{-\delta(n,p,k)}\|u_h\|_{L^2(M)}, 
\end{equation}
where
\begin{equation}\label{eqn:delta}
\delta(n,p,k)=\begin{cases}
   \frac{n-1}{4}-\frac{n-1}{2p} & 2\leq p \leq \frac{2(n+1)}{n-1} \\ \frac{n-1}{2} -\frac{n}{p} -\frac{1}{k+1}\left(\frac{n-1}{2}-\frac{n+1}{p} \right) & \frac{2(n+1)}{n-1}\leq p \leq \infty
\end{cases}.
\end{equation}
\end{theorem}

We note that for $2\leq p\leq \frac{2(n+1)}{n-1}$, the $L^p$ estimates from Theorem \ref{thrm:linfcontact} agree with those obtained by Sogge in (\ref{eqn:sogge}). In Example \ref{ex:sharp2} we will see that this agreement is because the saturating examples for Sogge's estimates for these values of $p$ can be found as joint quasimodes of operators $P_{1}$ and $P_{2}$ as described above. However, with the contact assumption for $\frac{2(n+1)}{n-1}<p\leq \infty$, we have a quantitative gain from (\ref{eqn:sogge}), which is illustrated in Figure \ref{fig:lpgraph}. This gain is most pronounced when $p=\infty$ and $k$ is small. Particularly, for $p=\infty$ and $k=1$, note that Theorem \ref{thrm:linfcontact} gives
\[
\|u_h\|_{L^\infty}\lesssim h^{-\frac{n-1}{4}}\|u_h\|_{L^2} \qquad \text{for} \quad k=1,
\]
while (\ref{eqn:sogge}) gives $\|u_h\|_{L^\infty}\lesssim h^{-\frac{n-1}{2}}\|u_h\|_{L^2}$. 
 \begin{figure}[ht] 
\begin{center}
\includegraphics[width=.5\textwidth]{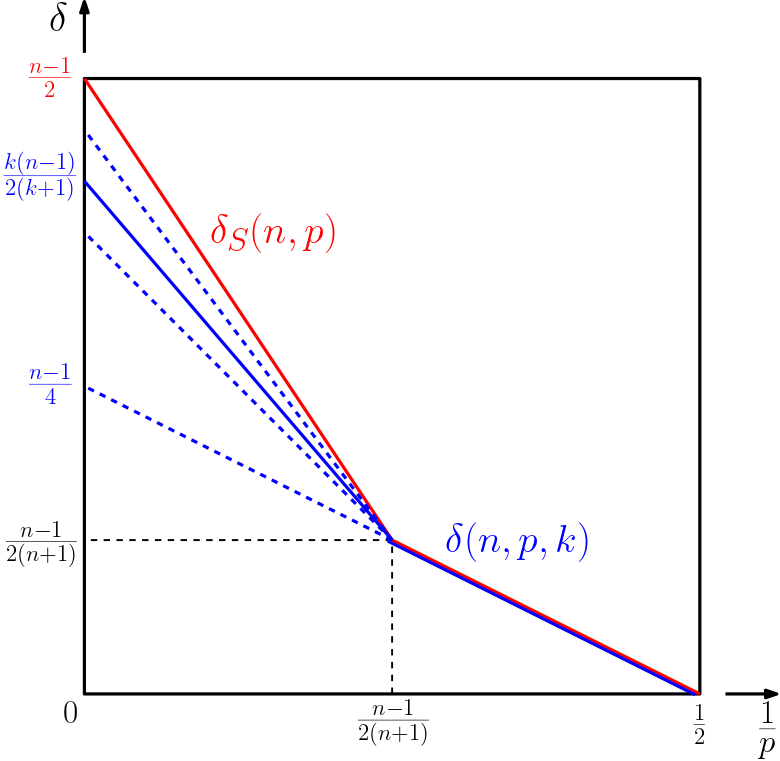} 
\end{center}
\caption{Comparison of $\delta(n,p,k)$ from Theorem \ref{thrm:linfcontact} (in \textcolor{blue}{blue}) for various values of $k$ with Sogge's estimate $\delta_{S}(n,p)$ from \cite{SoggeLp} (in \textcolor{red}{red}). } \label{fig:lpgraph}
\end{figure}
We will show in Examples \ref{ex:sharp} and \ref{ex:sharp2} in the next section that the bound given by (\ref{eqn:linf}) is sharp. Furthermore, note that Theorem \ref{thrm:linfcontact} recovers the result of \cite{TacyContact} if we set for $n=2$ .

\begin{remark}
    In the statement of Theorem \ref{thrm:linfcontact}, we assumed that there is only one point of contact. One can also handle multiple points of contact as long as they are at least $O(1)$ apart. To prove the theorem, we work in an $O(1)$ neighbourhood around the contact point and show $\|\chi u_h\|_{L^p}\lesssim h^{-\delta(n,p,k)}\|u_h\|_{L^2}$ where $\chi$ is supported in a small neighbourhood about the contact point. For multiple points of contact, the worst $L^p$ estimate would come from applying Theorem \ref{thrm:linfcontact} around the contact point with the largest order $k$. 
\end{remark}

The proof of Theorem \ref{thrm:linfcontact} follows a similar framework to \cite{TacyContact}. Without loss of generality, we will assume that the point of contact occurs at $(x_0,\xi_0)=(0,0)$. Next, we begin by using a Fourier integral operator $W$ to map the quasimode $u_h$ of $P_1$ and $P_2$ to a quasimode $v_h$ of $h{D_{x_1}}$ and some operator $q(x,hD)$. Moreover, using the structure of the contact, we can locally write
\[
q(x,\xi)\approx C|\xib|^{k+1},
\]
 where we denote $\xi=(\xi_1,\xi_2,\dots,\xi_n)=:(\xi_1,\xib)$.
Compared to $u_{h}$, it is easier to analyse $v_h$, as $v_h$ is a solution of the two ``simpler" equations, 
\[
hD_{x_1}v_h=O_{L^2}(h) \quad \text{and} \quad q(x,hD)v_h=O_{L^2}(h).
\]
Taking the semiclassical Fourier transform $\Fh$ gives 
\[
\xi_1 \Fh[v_h]=O_{L^2}(h) \quad \text{and} \quad  C|\xib|^{k+1} \Fh[v_h]=O_{L^2}(h),
\]
where
\[
\Fh[u](\xi)=\frac{1}{(2\pi h)^{\frac{n}{2}}}\int e^{-\frac{i}{h}\langle x,\xi \rangle}u(x) dx.
\]
Therefore, $\Fh[v_h]$ is concentrated in an $h$-region around $\xi_1=0$ and $|\xib|^{k+1}=0$. Hence, we expect $\Fh[v_h]$ to be approximately localised in an $h\times h^{\frac{n-1}{k+1}}$ region of phase space.  Using the structure of $W$ and the localisation of $\Fh[v_h]$ we are able to obtain the $L^p$ estimates on $u_h$.

\begin{remark}
    The assertion that $k$ is odd is related to the assumption that there is only one point of contact (in an $O(1)$ neighbourhood). Since there is only one point of contact, the hypersurfaces $\{p_j(0,\xi)=0\}$ do not pass through each other. If the order of contact were even, then roughly $q(x,\xi)$, which can be thought of as measuring the difference between the two hypersurfaces, behaves like $\xib_j^{k+1}$ where $k+1$ is odd. However, this implies that $q(x,\xi)$ vanishes at points other than $\xi=0$, which means that the hypersurfaces have more contact points. The details of this will become apparent in Section \ref{sec:setup}. In \cite{TacyContact}, the $n=2$ dimensional case, $k$ is allowed to be even, because for $n=2$, $q(x,\xi)\approx\xi^{k+1}$ for $\xi\in\R$ does not contradict the assumption that the hypersurfaces only meet at one point. 
\end{remark}

\subsection{Examples}\label{sec:examples}
We next consider a few examples to illustrate Theorem \ref{thrm:linfcontact}. Particularly, for characteristic function $\chi_h(\xi)$ we consider the functions
\begin{equation}\label{eqn:tchi}
T_{\chi_h}(x)=\frac{1}{(2\pi h)^{\frac{n}{2}}\|\chi_h\|_{L^2}}\int e^{\frac{i}{h}\langle x, \xi \rangle} \chi_h(\xi) d\xi=\frac{\Fh^{-1}[\chi_h](x)}{\|\chi_h\|_{L^2}}.
\end{equation}
It is easy to see that these functions are $L^2$-normalised since $\Fh$ is a unitary operator on $L^2$. 
\begin{example}[Sharpness of (\ref{eqn:linf}) for large $p$]\label{ex:sharp}
    Let 
    \[
    p_1(\xi)=\xi_1-|\xib|^2 \quad \text{and} \quad p_2(\xi)=\xi_1-\left(|\xib|^2-|\xib|^{k+1}\right),
    \]
    where $\xi=(\xi_1,\xib)$. Note that $p_1$ indeed satisfies the curvature condition. Furthermore $\{p_1=0\}$ and $\{p_2=0\}$ meet at $\xi=0$ with $k$-th order contact. To see this, observe that
    \[
    p_1-p_2=|\xib|^{k+1}
    \]
    and note that along any line through $\xi=0$ the derivatives of $|\xib|^{k+1}$ vanish to the $k$-th order.
    To saturate the $L^p$ estimates (for large $p$), one usually tries to spread the support $\Fh[u]$ through the largest possible region. Therefore, we define $\chi_h$ by
    \[
    \chi_h=\begin{cases}
        1 & |p_1|\leq h, \,\, |p_2|\leq h
        \\
        0 & \text{otherwise}
    \end{cases}
    \]
     and $T_{\chi_h}(x)$ be as defined in (\ref{eqn:tchi}). We see that $T_{\chi_h}(x)$ is a joint quasimode of order $h$ of $p_1(hD)$ and $p_2(hD)$ since on the support of $\chi_h$ we have $|p_1|,|p_2|\leq h$. Indeed, note
     \begin{align*}
     \left\|P_1(hD)^{M_1}P_2(hD)^{M_2}T_{\chi_h}(x)\right\|_{L^2} & =\frac{1}{\|\chi_h\|_{L^2}} \left\|P_1(hD)^{M_1}P_2(hD)^{M_2} \Fh^{-1}[\chi_h]\right\|_{L^2} \\
     & =\frac{1}{\|\chi_h\|_{L^2}} \left\|p_1^{M_1} p_2^{M_2} \chi_h\right\|_{L^2} \leq h^{M_1+M_2}\|T_{\chi_h}\|_{L^2}.
     \end{align*}
      We also see that for $|p_1|,|p_2|\leq h$, 
     \[
     |\xib|^{k+1}=|p_1-p_2|\leq 2h ,
     \]
     which implies that $|\xi_j|\lesssim h^{\frac{1}{k+1}}$ for $j=2,3,\dots,n$.
     Moreover, one can check that for
    \[
    |x_1|\ll h^{1-\frac{2}{k+1}}, \qquad  \text{and} \qquad |x_j|\ll h^{1-\frac{1}{k+1}} \,\, \text{for} \,\, j=2,3,\dots,n
    \]
    the phase $\frac{i}{h}\langle x, \xi \rangle$ in (\ref{eqn:tchi}) is not oscillating. Hence
    \[
    |T_{\chi_h}(x)|\gtrsim \frac{1}{h^{\frac{n}{2}}\|\chi_h\|_{L^2}} \|\chi_h\|_{L^1}=h^{-\frac{n}{2}} \sqrt{\operatorname{Vol}(\supp\chi_h)}.
    \]
    We compute $\operatorname{Vol}(\supp\chi_h) \simeq h^{1+\frac{n-1}{k+1}}$, and so 
    \[
    |T_{\chi_h}(x)|\gtrsim  h^{-\frac{n}{2}+\frac{1}{2}+\frac{n-1}{2(k+1)}} =h^{-\frac{n-1}{2}\left(1-\frac{1}{k+1}  \right)}.
    \]
   Therefore, the $L^\infty$ estimate provided in Theorem \ref{thrm:linfcontact} is sharp. Furthermore,
    \begin{align*}
    \|T_{\chi_h}(x)\|_{L^p} &\gtrsim  h^{-\frac{n-1}{2}\left(1-\frac{1}{k+1}  \right)}  \left( 
    \int \mathds{1}_{|x_1|\leq h^{1-\frac{2}{k+1}}} dx_1 \left(\int \mathds{1}_{|t|\leq h^{1-\frac{1}{k+1}} } dt \right)^{n-1}\right)^{1/p} \\ 
    &\gtrsim h^{-\frac{n-1}{2}\left(1-\frac{1}{k+1}  \right)} \left( h^{1-\frac{2}{k+1}} h^{\left(1-\frac{1}{k+1}\ \right)(n-1)}  \right)^{1/p} \\
    &=h^{-\frac{n-1}{2}\left(1-\frac{1}{k+1}  \right) +\frac{n}{p}-\frac{n+1}{p(k+1)}},
    \end{align*}
    which implies that the estimate given in (\ref{eqn:linf}) is sharp for $\frac{2(n+1)}{n-1}\leq p \leq \infty$. 
\end{example}

\begin{example}[Sharpness of (\ref{eqn:linf}) for small $p$]\label{ex:sharp2}
 As in the previous example, let 
  \[
    p_1(\xi)=\xi_1-|\xib|^2 \quad \text{and} \quad p_2(\xi)=\xi_1-\left(|\xib|^2-|\xib|^{k+1}\right).\]
 Define
    \[
    \chi_h=\begin{cases}
        1 & |p_1|\leq h, \,\, |p_2|\leq h, \,\,  |\xi_j|\lesssim h^{1/2} \,\, \text{for} \,\, j=2,3,\dots,n
        \\
        0 & \text{otherwise}
    \end{cases}
    \]
     and $T_{\chi_{h}}(x)$ be as defined in (\ref{eqn:tchi}). As in the previous example, $T_{\chi_{h}}(x)$ is a joint quasimode of order $h$ of $p_1(hD)$ and $p_2(hD)$ since on the support of $\chi$ we have $|p_1|,|p_2|\leq h$.
     Moreover for
    \[
    |x_1|\ll 1, \quad |x_j|\ll h^{1/2} \,\, \text{for} \,\, j=2,3,\dots,n
    \]
    the phase $\frac{i}{h}\langle x, \xi \rangle$ from (\ref{eqn:tchi}) is not oscillating and hence
    \[
    |T_{\chi_h}(x)|\gtrsim \frac{1}{h^{n/2}\|\chi_h\|_{L^2}} \|\chi_h\|_{L^1}=h^{-n/2} \sqrt{\operatorname{Vol}(\supp\chi_h)}.
    \]
    We calculate $\operatorname{Vol}(\supp\chi_h) \simeq h \times h^{\frac{n-1}{2}}$,
    and so 
    \[
    |T_{\chi_h}(x)|\gtrsim  h^{-\frac{n}{2}} h^{\frac{n-1}{4}+\frac{1}{2}}=h^{-\frac{n-1}{4}}.
    \]
Furthermore,
    \begin{align*}
    \|T_{\chi_h}(x)\|_{L^p} &\gtrsim h^{-\frac{n-1}{4}}  \left( 
    \int \mathds{1}_{|x_1|\leq 1} dx_1 \left( \int \mathds{1}_{ |t|\leq h^{\frac{1}{2}}} dt \right)^{n-1} \right)^{1/p} \\ 
    &\gtrsim h^{-\frac{n-1}{4}+\frac{n-1}{2p}} =h^{-\frac{n-1}{2}\left(\frac{1}{2}-\frac{1}{p} \right)},
    \end{align*}
      which implies that the estimate given in (\ref{eqn:linf}) is sharp for $2\leq p \leq \frac{2(n+1)}{n-1}.$ Note that these functions $T_{\chi_{h}}$ resemble the highest weight spherical harmonics (the low $p$ sharp examples for Sogge's \cite{SoggeLp} estimates). The existence of this family of examples precludes any improvement in the low $p$ estimates by adding an additional operator. 
\end{example}

The contact condition described in Theorem \ref{thrm:linfcontact} states that each pair of geodesics on the characteristic hypersurfaces meet with contact of order $k$, where the order of contact between two curves is given by the order in which their derivatives agree. As mentioned in Section \ref{sec:sor}, one could consider examples where the order of contact varies. In the following example, we illustrate what can go wrong if we widen our notion of contact to include different orders along different geodesics. 
\begin{example}(Varying orders of contact) Instead of assuming the characteristic hypersurfaces meet with $k$-th order contact along each line, in this example, we consider characteristic sets that meet with what we call $1,k$ type contact. That is, the characteristic sets meet with first-order contact along all but one line, which has $k$-th order contact.  
  For concreteness, set $n=3$. The ``model" case which exhibits this type of contact is the following: Let $p_1(\xi)=\xi_1-(\xi_2^2+\xi_3^{2})$ and $p_2(\xi)=\xi_1-(2\xi_2^2+\xi_3^2+\xi_3^{k+1})$. Then
    \[
    p_1-p_2=\xi_2^2+\xi_3^{k+1}
    \]
    has first order contact along all lines $\{\xi_3=m\xi_2\}$ for $m\in\R$, and along the line $\{\xi_2=0\}$ has $k$-th order contact. Next, let
    \[
    \chi_h(\xi)=\begin{cases}
        1 & |p_1|\leq h, \,\, |p_2|\leq h
        \\
        0 & \text{otherwise}
    \end{cases}
    \]
    and $T_{\chi_{h}}(x)$ be as defined in (\ref{eqn:tchi}). Once again, one can check that $T_{\chi_{h}}(x)$ is a joint quasimode for $p_1(hD)$ and $p_2(hD)$, and that $p_1$ satisfies the curvature condition. Moreover, on the support of $\chi_h$, we have
    \[
    \xi_2^2+\xi_3^{k+1}=|p_1-p_2|\leq 2h,
    \]
    which implies $|\xi_2|\lesssim h^{1/2}$ and $|\xi_3|\lesssim h^{\frac{1}{k+1}}$. As in the previous two examples, we can show that
    \[
    \operatorname{Vol}(\supp\chi_h)\simeq h\times h^{\frac{1}{2}}\times h^{\frac{1}{k+1}},
    \]
    and so,
    \[
    |T_{\chi_h}(x)|\gtrsim h^{-\frac{3}{2}}\sqrt{\operatorname{Vol}(\supp\chi_h)}=h^{-\frac{3}{4}+\frac{1}{2(k+1)}}.
    \]
   We call this the ``model" case of $1,k$ type contact because both orders of contact naturally appear in the volume of the support of $\chi_h$. Unfortunately, this is not true for all $1,k$ type contact. Consider
   \[
   q_1(\xi)=\xi_1-(\xi_2^2+\xi_3^{2}) \quad \text{and} \quad q_2(\xi)=\xi_1-(\xi_2^2+\xi_3^{2}-(\xi_2-\xi_3^2)^2-\xi_2^{10}) 
   \]
   and
   \[
     \kappa_h(\xi)=\begin{cases}
        1 & |q_1|\leq h, \,\, |q_2|\leq h
        \\
        0 & \text{otherwise}
    \end{cases}.
   \]
   Let $T_{\kappa_h}(x)$ be as defined in (\ref{eqn:tchi}).  We note that $T_{\kappa_{h}}(x)$ is a joint quasimode for $q_1(hD)$ and $q_2(hD)$, and moreover $q_1$ satisfies the curvature condition. Furthermore, note that along any line $\{\xi_3=m \xi_2\}$ we have
   \[
   q_1-q_2\big|_{\xi_3=m \xi_2}=(\xi_2-\xi_3^2)^2+\xi_2^{10}\big|_{\xi_3=m \xi_2}=\xi_2^2-2 m^2 \xi_2^3+m^4 \xi_2^4+\xi_2^{10},
   \]
   which has vanishing derivatives to the first order. On the other hand, along the line $\xi_2=0$, we note
   \[
   q_1-q_2\big|_{\xi_2=0}=\xi_3^4,
   \]
   which has vanishing derivatives to the third order. So $T_{\kappa_h}(x)$ is an example of $1,3$ type contact. We might expect based on the computation for $\chi_h$ that $|T_{\kappa_h}(x)|\gtrsim h^{-\frac{3}{4}+\frac{1}{8}}$. However, we compute
   \[
   \operatorname{Vol}(\supp\kappa_h)\simeq h \times h^{\frac{1}{2}}\times h^{\frac{1}{20}},
   \]
   and hence,
   \[
   |T_{\kappa_h}|\gtrsim h^{-\frac{3}{4}+\frac{1}{40}}.
   \]
   This lower bound is different from the ``model" example because a factor of $h^{1/20}$ appears in the volume, which has no relation to how we have defined the order of contact. This factor of $h^{1/20}$ occurs because the function $q_1-q_2=(\xi_2-\xi_3^2)^2+\xi_2^{10}$ vanishes along the curve $\xi_2=\xi_3^2$. Therefore, on this curve we have $q_1-q_2\approx \xi_2^{10}=\xi_3^{20}$, and so the ``order of contact" is 19th order. Unfortunately, our notion of contact does not allow us to ``see" this higher-order ``contact" along the parabola $\xi_2=\xi_3^2$. To remedy this, we could update our notion of contact to include all unit-speed paths, not just straight lines/geodesics, through the point of contact. We plan to address this in future work. 
\end{example}

\section{Simplifying Equations Using a Fourier Integral Operator}\label{sec:setup}
In this section, we set up the notation for the remainder of the paper. Mainly, we define the Fourier integral operator $W$, which allows us to map the quasimode $u_h$ to a quasimode of two simpler equations. The key idea is that $W$ will ``flatten'' out the characteristic set of $p_{1}(x,\xi)$ so that it becomes the hyperplane $\{\xi_{1}=0\}$ as depicted in Figure \ref{fig:flatten}. This characteristic set is associated with the differential operator $hD_{x_{1}}$ and so if $v_{h}=Wu_{h}$ we will see that $v_{h}$ is a quasimode for the operator $hD_{x_{1}}$ (Lemma \ref{lem:vquasi}). This means that $v_{h}$ is roughly constant in the $x_{1}$ direction. Using Egorov's Theorem we will further show that, since $u_{h}$ is also a quasimode of $P_{2}$, $v_{h}$ is a quasimode of a second operator $q(x,hD)$, and that the contact condition implies
\[|q(x,\xi)|\geq c(x)|\bar{\xi}|^{k+1},\]
where $\xi=(\xi_{1},\bar{\xi})$ (Lemma \ref{lem:vjquasi}).

\begin{figure}
\centering
\begin{subfigure}{.48\textwidth}
  \centering
  \includegraphics[width=.95\linewidth]{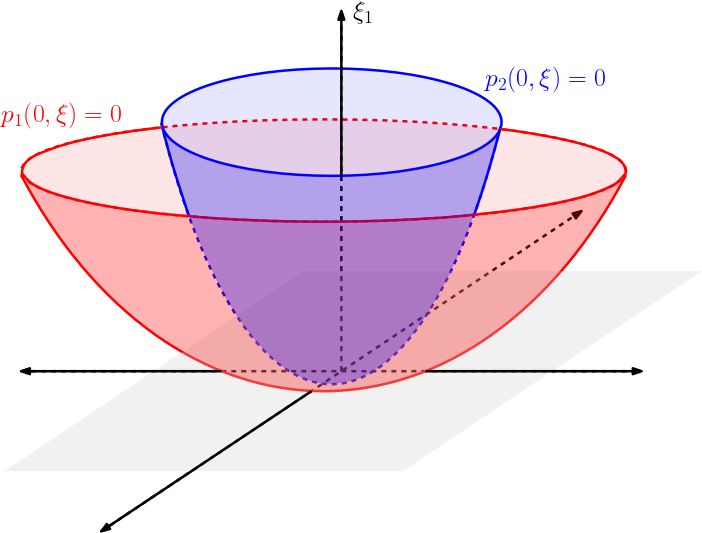}
\end{subfigure}%
\,\,
\begin{subfigure}{.48\textwidth}
  \centering
  \includegraphics[width=.95\linewidth]{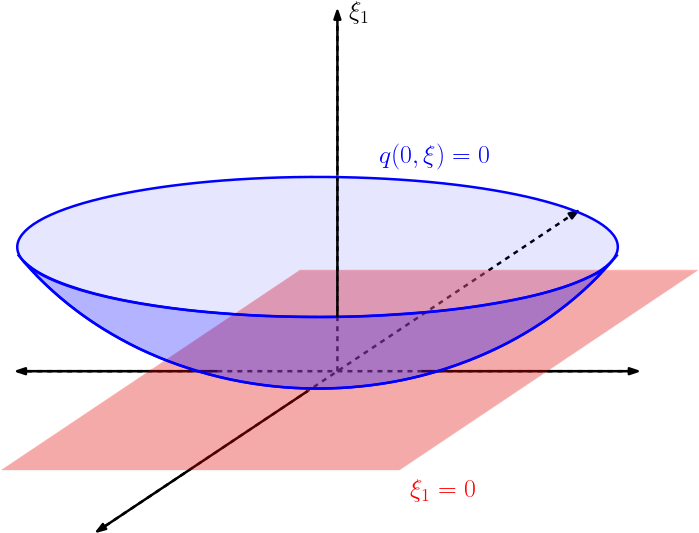}
\end{subfigure}
\caption{A diagram of $k$th order contact between the two characteristic sets $\{p_1=0\}$ and $\{p_2=0\}$ (left) and the resulting hypersurfaces after flattening $\{p_1=0\}$ using the $W$ operator (right). }
\label{fig:flatten}
\end{figure}

\subsection{Localisation and choosing coordinates} Let $(x_0,\xi_0)$ denote the point of contact, which satisfies
\[
\{(x_0,\xi_0)\}= \bigcap_{j=1}^2\{\xi: p_j(x_0,\xi)=0\}.
\]
Without loss of generality, we can, by translation, assume $(x_0,\xi_0)=(0,0)$. It will be sufficient to estimate $\|\chi(x,hD)u_h\|_{L^p}$ where $\chi\in\cci$ is supported in a sufficiently small (but not $h$-dependent) neighbourhood of $(0,0)$. To see this, first note that $\chi(x,hD)u_h$ is still a strong joint quasimode of order $h$ for $P_1$ and $P_2$. Indeed, for $j=1,2$ we can write
\[P_j^M \chi(x,hD)u_h=\sum_{\ell=0}^M\binom{M}{\ell}h^\ell R_\ell P_j^{M-\ell}u_h,\]
where $R_0=\chi(x,hD)$ and $R_\ell=\frac{1}{h}[P_j,R_{\ell-1}]$. Moreover, since the commutator of two pseudodifferential operators satisfies 
\[
[P_j,R_{\ell-1}]=\frac{h}{i}\{ p_j,r_{\ell-1}\}(x,hD)+O_{L^2\to L^2}(h^2),
\]
we know that $\|R_\ell\|_{L^2\to L^2}=O(1)$. Hence, using that $u_h$ is a strong quasimode for $P_j$ we have
\[
\|P_j^M \chi(x,hD)u_h \|_{L^2}\leq \sum_{\ell=0}^M\binom{M}{\ell}h^\ell \|R_\ell\|_{L^2\to L^2} \|P_j^{M-\ell}u_h\|_{L^2}\lesssim h^M \|u_h\|_{L^2}.
\]
Next, note that if $\chi\in\cci$ is supported sufficiently close to $\rho_0$ where $p_j(\rho_0)\not=0$ for some $j$, then on the support of $\chi$, $P_j$ is invertible. Moreover, $P_j^M$ is invertible. Using that $\chi(x,hD) u_h$ is a strong quasimode of $P_j$ we know
\[
P_j^M \chi(x, hD)u_h =h^M f,
\]
where $\|f\|_{L^2}\leq \|u_h\|_{L^2}$. Since $P_j^M$ has an inverse, we have
\[
\chi(x, hD)u_h=h^M (P_j^M)^{-1}f
\]
and so, 
\[
\|\chi(x,hD)u_h\|_{L^2}\lesssim h^M\|u_h\|_{L^2}.
\]
Applying semiclassical Sobolev estimates, we obtain
\[
\|\chi(x,hD)u_h\|_{L^p}\lesssim h^{\frac{n}{p}-\frac{n}{2}+M}\|u_h\|_{L^2}.
\]
Choosing $M$ large enough, we obtain a better estimate than those provided in Theorem \ref{thrm:linfcontact}, and thus we can ignore these negligible contributions. Therefore, we only need to consider $\chi(x,hD)u_h$ where $\chi(x,\xi)$ is supported sufficiently close to $(0,0).$ Furthermore, since we only need to consider $u_h$ localised in phase space, it suffices to work on a coordinate patch of $M$ associated to a neighbourhood of $0$ in $\R^n$. Thus, we work over the patch in $\R^n$ for the remainder of the paper.

Next, since $\{\xi:p_1(0,\xi)=0\}$ is a smooth hypersurface we know
\[
\grad_{\xi}p_1(0,0)\not=0.
\]
After possibly rotating coordinates, we may assume 
\[
\partial_{\xi_1}p_1(0,0)\not=0 \quad \text{and} \quad \grad_{\xib}p_1(0,0)=0,
\]
where we write $\xi=(\xi_1,\xib)$.
Furthermore, we can assume that on the support of $\chi$ (possibly shrinking the support of $\chi$ if needed) that we have $|\partial_{\xi_1}p_1(x,\xi)|>c>0$. Therefore, by the Implicit Function Theorem, we may write
\[
p_1(x,\xi)=e_1(x,\xi)(\xi_1-a_1(x,\xib)),
\]
where $|e_1(x,\xi)|>0$ and hence $e_1(x,hD)$ is invertible. Since $u_h$ is a quasimode of $P_1$ and since $e_1(x,hD)$ is elliptic, $u_h$ also satisfies
\[
(hD_{x_1}-a_1(x, hD_{\xb}))u_h=O_{L^2}(h \|u_h\|_{L^2}),
\]
where we are using the coordinates $x=(x_1,x_2,\dots,x_n)=:(x_1,\xb)$.
The curvature condition of $p_1$ implies that
\[
\partial_{\xib}^2p_1(0,0) \, \text{ is non-degenerate}.
\]
Next, since the characteristic sets are tangential, we must also have that $\grad_{\xib}p_2(0,0)=0$. However, since $\{\xi:p_2(0,\xi)=0\}$ is also a smooth hypersurface, we also know $\grad_{\xi}p_2(0,0)\not=0.$ Thus, as we had for $p_1$, it must be that $\partial_{\xi_1}p_2(0,0)\not=0$. Once again, using the Implicit Function Theorem, we can also factorise $p_2$ as
\[
p_2(x,\xi)=e_2(x,\xi)(\xi_1-a_2(x,\xib)),
\]
where $|e_2|>0$. Therefore, using the ellipticity of $e_2(x,hD)$ we have
\[
(hD_{x_1}-a_2(x, hD_{\xb}))u_h=O_{L^2}(h \|u_h\|_{L^2}).
\]
Therefore, we see that $u_h$ is an order $h$ quasimode of the two slightly easier operators to work with, $(hD_{x_1}-a_1(x, hD_{\xb}))$ and $(hD_{x_1}-a_2(x, hD_{\xb}))$.
\subsection{A Fourier integral operator} Next, we introduce the Fourier integral operator $W$ that ``flattens" the characteristic set associated to $hD_{x_1}-a_1(x, hD_{\xb})$. 
Let $W(x_1):L^2(\R^{n-1})\to L^2(\R^{n-1})$ be the operator so that
\begin{equation}\label{eqn:w1def}
\begin{cases}
h D_{x_1} W(x_1)+W(x_1) a_1(x_1, \bar{x}, h D_{\bar{x}})=O_{L^2}(h^\infty) \\
W(0)=\operatorname{Id}
\end{cases}
\end{equation}
for $|x_1|$ sufficiently small, where $\operatorname{Id}$ denotes the identity operator on $L^2(\R^{n-1})$. 
It is known that the operator $W(x_1)$ is unitary and can be represented using an oscillatory integral (see \cite[Theorem 10.4]{zworski}). We have
\begin{equation}\label{eqn:w1int}
[W(x_1) f](\bar{y}) =\frac{1}{(2 \pi h)^{n-1}} \int_{\R^{n-1}} \int_{\R^{n-1}} e^{\frac{i}{h} (\langle \bar{y}, \bar{\eta} \rangle -\varphi_1(x_1,\bar{x},\bar{\eta}))} b^h(x,\bar{\eta}) f(\bar{x}) d\bar{x}  d\bar{\eta}, \quad  \bar{y}\in \R^{n-1},
\end{equation}
where $\varphi_1$ satisfies
\begin{equation}\label{eqn:phi1def}
\begin{cases}
\partial_{x_1} \varphi_1 -a_1(x_1,\bar{x},\partial_{\bar{x}} \varphi_1)=0 \\
\varphi_1(0,\bar{x},\bar{\eta})=\langle \bar{x}, \bar{\eta} \rangle
\end{cases}
\end{equation}
and $b^h(x_1,\cdot,\cdot)\in \cci (\R^{n-1}\times \R^{n-1})$ satisfies
\[
b^h(x_1,\xb,\xib) \sim \sum_{j=0}^\infty h^j b_j(x_1,\xb,\xib)
\]
and $b_0(0,\xb,\xib)=1$.
Note that (\ref{eqn:w1def}) is equivalent to
\begin{equation}\label{eqn:w*def}
(h D_{x_1}-a_1(x_1,\bar{x},hD_{\bar{x}})) W^*(x_1)=O_{L^2}(h^\infty).
\end{equation} 
To see this, note that since $W(x_1)$ is unitary we have
\begin{align*}
  0=  hD_{x_1}\left( W(x_1) W^*(x_1) \right) &= \left(hD_{x_1}W(x_1) \right) W^*(x_1) +W(x_1) \left(hD_{x_1}W^*(x_1) \right) \\
  &=W(x_1)\left(hD_{x_1}-a_1(x,hD_{\xb}) \right) W^*(x_1)+O(h^\infty),
\end{align*}
where we used (\ref{eqn:w1def}). Multiplying by $W^*(x_1)$ on the left implies (\ref{eqn:w*def}).

Next, define
\begin{equation}\label{eqn:vdef}
v_h(x_1,\bar{y}):=W(x_1)u_h(x_1, \cdot)=\frac{1}{(2 \pi h)^{n-1}} \int_{\R^{n-1}} \int_{\R^{n-1}} e^{\frac{i}{h} (\langle \bar{y}, \bar{\eta} \rangle -\varphi_1(x_1,\bar{x},\bar{\eta}))} b^h(x_1,\bar{x},\bar{\eta}) u_h(x_1, \bar{x}) d\bar{x}  d\bar{\eta}.
\end{equation}
The following lemma shows that $v_h$ is a strong quasimode of $h D_{x_1}$ of order $h$. 

\begin{lemma}\label{lem:vquasi}
    Let $W(x_1)$ and $v_h$ be given by (\ref{eqn:w1int}) and (\ref{eqn:vdef}) respectively. Then $v_h$ is a strong quasimode of order $h$ for $hD_{x_1}$. 
 \end{lemma} 
  \begin{proof}
First, using (\ref{eqn:w1def}), we note
\begin{align*}
h D_{x_1} v_h= hD_{x_1} (W(x_1) u_h(x_1,\bar{x})) &= (hD_{x_1}W(x_1)) u_h + W(x_1) hD_{x_1}u_h \\
&= W(x_1) \left( hD_{x_1} - a_1(x_1,\bar{x},hD_{\bar{x}})  \right)u_h +O_{L^2}(h^\infty).
\end{align*}
Next assume for some $M\in\N$ we have $(h D_{x_1})^M v_h =W(x_1)(h D_{x_1}-a_1(x_1,\bar{x},hD_{\bar{x}}))^M u_h+O_{L^2}(h^\infty)$. Consider 
\[
(h D_{x_1})^{M+1} v_h=h D_{x_1}(h D_{x_1})^M v_h =h D_{x_1} (W(x_1)(h D_{x_1}-a_1(x_1,\bar{x},hD_{\bar{x}}))^M u_h) +O_{L^2}(h^\infty). 
\]
Applying the product rule we have
\begin{align*}
(h &D_{x_1})^{M+1} v_h \\
&=(h D_{x_1}W(x_1))(h D_{x_1}-a_1(x_1,\bar{x},hD_{\bar{x}}))^M u_h + W(x_1) h D_{x_1}((h D_{x_1}-a_1(x_1,\bar{x},hD_{\bar{x}}))^M u_h)) +O_{L^2}(h^\infty)  \\
&=W(x_1)(h D_{x_1}-a_1(x_1,\bar{x},hD_{\bar{x}}))^{M+1} u_h +O_{L^2}(h^\infty) .
\end{align*}
Thus, by induction, we have for all $M\in\N$, 
\begin{equation}\label{eqn:induction}
(h D_{x_1})^M v_h =W(x_1)(h D_{x_1}-a_1(x_1,\bar{x},hD_{\bar{x}}))^M u_h +O_{L^2}(h^\infty).
\end{equation}
Therefore, since $W(x_1)$ is unitary we have
\[
\| (h D_{x_1})^M v_h \|_{L_{\bar{y}}^2} \leq \| (h D_{x_1}-a_1(x_1,\bar{x},hD_{\bar{x}}))^M u_h\|_{L_{\bar{x}}^2}.
\]
Next, since $u_h$ is a strong quasimode for $(h D_{x_1}-a_1(x_1,\bar{x},hD_{\bar{x}}))$ and using the localisation we have imposed on $u_h$ we have
\[
\|(h D_{x_1})^M v_h \|_{L^2} \lesssim h^M \|u_h\|_{L^2},
\]
as desired.
\end{proof}
\subsection{Incorporating the second equation}
In the previous subsection we saw that the operator $W(x_1)$ maps the quasimode $u_h$ of $hD_{x_1}-a_1(x,hD_{\xb})$ to a quasimode $v_h$ of $h D_{x_1}$:
\[
\|(hD_{x_1}-a_1(x,hD_{\xb})^M u_h\|_{L^2}\lesssim h^M\|u_h\|_{L^2} \quad \underset{W(x_1)}{\implies} \quad   \| (hD_{x_1})^Mv_h\|_{L^2}\lesssim h^M\|u_h\|_{L^2}.
\]
We also need to convert the information that $u_h$ is a quasimode for $hD_{x_1}-a_2(x,hD_{\xb})$ over to $v_h$ using $W(x_1)$. Notice that since $u_h=W^*(x_1)v_h$, and using that $u_h$ is also a quasimode of $hD_{x_1}-a_2(x,hD_{\xb})$, we have
\[
\|W(x_1)(hD_{x_1}-a_2(x,hD_{\xb})) W^*(x_1) v_h\|_{L^2}\lesssim \|(hD_{x_1}-a_2(x,hD_{\xb})) u_h\|_{L^2}\lesssim h \|u_h\|_{L^2}.
\]
Therefore, $v_h$ is also a quasimode of
\begin{equation}\label{eqn:qopdef}
Q=q(x,hD_{\xb}):=W(x_1)(hD_{x_1}-a_2(x,hD_{\xb})) W^*(x_1).
\end{equation}
Furthermore, note that since $W$ satisfies (\ref{eqn:w*def}) we have
\[Q=W(x_1)(a_1(x,hD_{\xb})-a_2(x,hD_{\xb})) W^*(x_1)+O(h^\infty).\]
We will use Egorov's Theorem (see, for example, \cite[Theorem 11.1]{zworski}) to understand $Q$ further. Egorov's Theorem states that quantisation given by the conjugation with $W(x_1)$ follows the Heisenberg picture of quantum mechanics. That is, $W(x_1)(hD_{x_1}-a_2(x,hD_{\xb})) W^*(x_1)$ is the quantisation of the classical observable $\kappa_{x_1}^*(a_1-a_2)$ up to an error of size $h$, where $\kappa_{x_1}$ is the classical flow given by:
\[
\kappa_{x_1}(x,\xib)\mapsto (x_1,\xb(x_1),\xib(x_1)), \quad \text{where} \quad \begin{cases}
\frac{d\bar{x}(x_1)}{dx_1}=-\grad_{\bar{\xi}}a_1(x_1,\bar{x},\bar{\xi}) \\
\bar{x}(0)=\bar{x} \\
\frac{d\bar{\xi}(x_1)}{dx_1}=\grad_{\bar{x}}a_1(x_1,\bar{x},\bar{\xi}) \\
\bar{\xi}(0)=\bar{\xi} 
\end{cases}.
\]
Using this flow, we define
\begin{equation}\label{eqn:tildeqdef}
\tilde{q}(x,\xib):= a_1(x_1,\xb(x_1),\xib(x_1))-a_2(x_1,\xb(x_1),\xib(x_1))=\kappa_{x_1}^*(a_1-a_2).
\end{equation}
Then Egorov's Theorem implies
\[
Q=W(x_1)(a_1(x,hD_{\xb})-a_2(x,hD_{\xb})) W^*(x_1) +O(h^\infty)=\tilde{q}(x,hD_{\xb}) + O_{L^2\to L^2}(h).
\]
Since $v_h$ is a quasimode of $Q$, it will also be a strong quasimode of $\tilde{q}(x,hD_{\xb})$, which we show in the following lemma. Furthermore, we will also show that $\tilde{q}(x,\xib)$ has a particular structure. 

\begin{lemma}\label{lem:vjquasi}
    Let $\tilde{q}$ be as defined in (\ref{eqn:tildeqdef}) and $v_h$ as defined in (\ref{eqn:vdef}). Then $v_h$ is a strong joint quasimode of order $h$ of $hD_{x_1}$ and $\tilde{q}(x,hD_{\xb})$. Moreover, on the support of $\chi$ (possibly shrinking the support if needed), we have
    \begin{equation}\label{eqn:qstructure}
    \tilde{q}(x,\xib)\geq c(x)|\xib|^{k+1},
    \end{equation}
    where $c(x)>0$ on the support of $\chi$. 
\end{lemma}
\begin{proof} In Lemma \ref{lem:vquasi} we showed that $v_h$ is a strong quasimode of order $h$ of $hD_{x_1}$. 
    We have also seen that $v_h$ is a quasimode of $Q=W(x_1)(hD_{x_1}-a_2(x,hD_{\xb})) W^*(x_1)$. Furthermore we can express $\tilde{q}(x,hD_{\xb})=Q+hR$ where $R:L^2\to L^2$ is bounded (independent of $h$).
    
    We first show that $v_h$ is a strong quasimode of order $h$ for $\tilde{q}(x,hD_{\xb})$. Let $N\in\N$ and note that we can write
    \[ \tilde{q}(x,hD_{\xb})^Nv_h =(Q+hR)^N v_h = \sum_{j=0}^N h^j R_j(x,hD_{\xb}) Q^{N-j}v_h, \]
     where all the pseudodifferential operators $R_j:L^2\to L^2$ are bounded independent of $h$. Further, using the definition of $Q$ from (\ref{eqn:qopdef}) we have
   \begin{align*}
       \tilde{q}(x,hD_{\xb})^Nv_h  & =\sum_{j=0}^N h^j R_j(x,hD_{\xb}) \big( W(x_1)(hD_{x_1}-a_2(x,hD_{\xb})) W^*(x_1) \big)^{N-j}v_h \\
        &=\sum_{j=0}^N h^j R_j(x,hD_{\xb})  W(x_1)\big(hD_{x_1}-a_2(x,hD_{\xb})\big)^{N-j} W^*(x_1) v_h.
   \end{align*}
  Moreover, since $u_h=W^*(x_1)v_h$ and using that $u_h$ is a strong quasimode of order $h$ for $hD_{x_1}-a_2(x,hD_{\xb})$ we have
   \[
   \|\tilde{q}^N(x,hD_{\xb})v_h\|_{L^2} \lesssim \sum_{j=0}^N h^j \| W(x_1)\big(hD_{x_1}-a_2(x,hD_{\xb})\big)^{N-j} u_h \|_{L^2} \lesssim h^N\|u_h\|_{L^2},
   \]
   where we used the unitarity of $W(x_1).$ Thus indeed $v_h$ is a strong quasimode for $\tilde{q}(x,hD_{\xb})$. Furthermore, using (\ref{eqn:induction}) we see
   \begin{align*}
       \tilde{q}&(x,hD_{\xb})^N(hD_{x_1})^M v_h  \\
       & =  \tilde{q}(x,hD_{\xb})^N W(x_1) (hD_{x_1}-a_1(x,hD_{\xb}))^M u_h +O(h^\infty) \\
       &=\sum_{j=0}^N h^j R_j(x,hD_{\xb})  W(x_1)\big(hD_{x_1}-a_2(x,hD_{\xb})\big)^{N-j}  (hD_{x_1}-a_1(x,hD_{\xb}))^M u_h +O(h^\infty),
   \end{align*}
    where, once again, all the pseudodifferential operators $R_j:L^2\to L^2$ are bounded independent of $h$.
   Therefore, since $u_h$ is a strong quasimode of $hD_{x_1}-a_1(x,hD_{\xb})$ and $hD_{x_1}-a_2(x,hD_{\xb})$ we have
   \[
   \|\tilde{q}(x,hD_{\xb})^N (hD_{x_1})^M v_h \|_{L^2}\lesssim \sum_{j=0}^N h^j\|\big(hD_{x_1}-a_2(x,hD_{\xb})\big)^{N-j}  (hD_{x_1}-a_1(x,hD_{\xb}))^M u_h \|_{L^2}\lesssim h^{N+M}\|u_h\|_{L^2}.
   \]
   This shows that $v_h$ is indeed a strong joint quasimode of $hD_{x_1}$ and $\tilde{q}(x,hD_{\xb})$ as claimed.
   
   Next, we use the contact condition to understand further the structure of 
   \[\tilde{q}(x,\xib)=a_1(x_1,\xb(x_1),\xib(x_1))-a_2(x_1,\xb(x_1),\xib(x_1)).\]
   The contact condition on $p_1$ and $p_2$ implies that $a_1$ and $a_2$ have $k$-th order contact in all directions at $(0,0)$. At $(0,0)$, the first $k$ derivatives of $a_1$ and $a_2$ in each direction agree, and the $k+1$-th derivative differs. Therefore, we will consider all lines passing through $\xi=0$ to understand the contact condition further. We begin by introducing our notation for these lines. Fix some $j\in\{2,\dots,n\}$ and let $\bm^j=(m^{j}_2,m^{j}_3,\dots,m^{j}_n)\in\R^{n-1}$ such that $m^{j}_j=1$. Then any line, except those in the hyperplane $\xi_j=0$, can be written as $\bm^j \xi_j=(m^{j}_2\xi_j,\dots,m^{j}_{j-1}\xi_j,\xi_j,m^{j}_{j+1}\xi_j,\dots,m^{j}_n \xi_j)$. Iterating over $j$ recovers all lines passing through $\xib=0$. Therefore, we can write the contact condition for  $a_1$ and $a_2$  as follows:
   \[
   \frac{d^s}{d\xi_j^s}(a_1-a_2)(0,\bm^j\xi_j)\big|_{\xi_j=0}=0 \quad \text{for all} \,\, s\leq k \,\, \text{and for any} \,\, j \,\, \text{and} \,\, \bm^j 
   \]
   and
   \[
   \frac{d^{k+1}}{d\xi_j^{k+1}}(a_1-a_2)(0,\bm^j\xi_j)\big|_{\xi_j=0}\not=0 \quad \text{for any}  \,\, j \,\, \text{and} \,\, \bm^j . 
   \]
   Additionally, since we have assumed (on the support of $\chi$) that there is exactly one contact point, we know that $a_1-a_2$ is sign definite. Without loss of generality, assume $a_1-a_2\geq 0$ on the support of $\chi.$ If $a_1-a_2\leq 0$ we could study $-\tilde{q}$ instead. 
   
   To prove (\ref{eqn:qstructure}) we will Taylor expand $\tilde{q}(x,\xib)$ around $(0,0)$. Before doing so, we will set up our notation and make two observations. 
   
   \underline{Multiindex Notation 1:} Let $J$ be an $s$-tuple of integers between $2$ and $n$; $J=(j_1,j_2,\dots,j_s)$ and $J\in(\{2,\dots,n\})^{s}$ where $|J|=s$. We denote 
   \[
   \partial^{J}f(\xib)=\partial_{j_1}\partial_{j_2}\cdots\partial_{j_s}f(\xib) \quad \text{where} \quad \partial_{j_i}=\frac{\partial}{\partial_{\xi_{j_i}}}.
   \]
   and
   \[
   \xib^{J}=\xi_{j_1}\xi_{j_2}\cdots \xi_{j_s}.
   \]
   For example, for $J=(2,2,3)$ we have $|J|=3$, $\partial^{J}f(\xib)=\partial_{\xi_2}^2 \partial_{\xi_3}f(\xib)$, and $\xib^J=\xi_2^2\xi_3.$

   \underline{Multiindex Notation 2:} Let $\alpha$  be an $n-1$-tuple of natural numbers; $\alpha=(\alpha_2,\alpha_3,\dots,\alpha_n)\in \N^{n-1}$. We denote $|\alpha|=\alpha_2+\dots+\alpha_n$, $\alpha!=\alpha_2!\alpha_3!\cdots \alpha_n!$, 
   \[
   \partial^\alpha f(\xib)=\partial_{\xi_2}^{\alpha_2} \partial_{\xi_3}^{\alpha_3} \cdots \partial_{\xi_n}^{\alpha_n} f(\xib)  \quad \text{and} \quad \xib^\alpha=\xi_2^{\alpha_2} \xi_3^{\alpha_3} \cdots \xi_n^{\alpha_n}.
   \]
   For example, for $\alpha=(2,1,0,\dots,0)$ we have $\partial^\alpha f(\xib)=\partial_{\xi_2}^2 \partial_{\xi_3}f(\xib)$ and $\xib^\alpha=\xi_2^2 \xi_3$. We can switch from ``$J$-multiindex" notation to ``$\alpha$-multiindex" notation by setting $\alpha_\ell=\{\#i:j_i=\ell\}$. For example $J=(2,2,2,3,4,4)$ corresponds to $\alpha=(3,1,2,0,\dots,0)$. We will switch between ``$J$-multiindex" notation and ``$\alpha$-multiindex" notation depending on which is more convenient. 
   
   \underline{Claim 1:} Let $J$ be a multiindex as above. Then
   \begin{equation}\label{eqn:claim1}
   \partial^{J}(a_1-a_2)(0,0)=0 \quad \text{for any} \,\, |J|\leq k.
   \end{equation}
   To see this, consider $(a_1-a_2)$ along the line $(\xi_2, \xi_3=m^{2}_3\xi_2, \dots, \xi_n=m^{2}_n \xi_2)=\bm^2 \xi_2$. We will show that for any $0\leq s \leq k$
   \begin{equation}\label{eqn:Jderiv}
   \partial_{\xi_2}^s(a_1-a_2)(x,\bm^2\xi_2)=\sum_{|J|=s}(\bm^2)^J \partial^{J}(a_1-a_2)(x,\bm^2\xi_2),
   \end{equation}
   and hence evaluating at $x=0,\xi_2=0$ gives
   \[
   0=\partial_{\xi_2}^s(a_1-a_2)(0,0)=\sum_{|J|=s}(\bm^2)^J \partial^{J}(a_1-a_2)(0,0) \quad \text{for any}\,\, s\leq k,
   \]
   where we used the contact condition to conclude $0=\partial_{\xi_2}^s(a_1-a_2)(0,0).$ Then, since this polynomial, $\sum_{|J|=s}(\bm^2)^J \partial^{J}(a_1-a_2)(0,0)$, in the variables $m^{2}_3,\dots,m^{2}_n$,  vanishes identically, we must have that each coefficient is zero. This implies that for $s\leq k$
   \[
   \partial^{J}(a_1-a_2)(0,0)=0 \quad \text{for any} \,\, |J|\leq s,
   \]
   which proves the claim. It remains to prove (\ref{eqn:Jderiv}). We do so with induction on $s$. First we observe that
   \[
   \partial_{\xi_2}(a_1-a_2)(x,\bm^2\xi_2)=\sum_{j=2}^n m_j^2 \left(\partial_{\xi_j}a_1-\partial_{\xi_j}a_2 \right) (x,\bm^2\xi_2)=\sum_{|J|=1}(\bm^2)^J \left(\partial^J(a_1-a_2)\right)(x,\bm^2\xi_2).
   \]
    Next, assume (\ref{eqn:Jderiv}) holds for $0\leq s < k$ and consider
   \[
    \partial_{\xi_2}^{s+1}(a_1-a_2)(x,\bm^2\xi_2)=\partial_{\xi_2}\left(\sum_{|J|=s}(\bm^2)^J \partial^{J}(a_1-a_2)(x,\bm^2\xi_2) \right)=\sum_{|J|=s}(\bm^2)^J \partial_{\xi_2}\left( \partial^{J}(a_1-a_2)(x,\bm^2\xi_2) \right).
   \]
   Moreover, using the chain rule, we have
   \[
   \partial_{\xi_2}\left( \partial^{J}(a_1-a_2)(x,\bm^2\xi_2) \right)=\sum_{\ell=2}^n m^{2}_\ell \left(\partial_{\xi_\ell} \partial^{J}(a_1-a_2)\right)(x,\bm^2\xi_2)=\sum_{|L|=1}(\bm^2)^{L} \left(\partial^L \partial^J (a_1-a_2)\right)(x,\bm^2\xi_2).
   \]
    Thus,
    \begin{multline*}
        \partial_{\xi_2}^{s+1}(a_1-a_2)(x,\bm^2\xi_2)=\sum_{|J|=s} \sum_{|L|=1}(m^2)^{J+L} \left(\partial^{J+L}(a_1-a_2)\right)(x,\bm^2\xi_2) \\=\sum_{|J|=s+1}(\bm^2)^J \left( \partial^J (a_1-a_2)\right)(x,\bm^2 \xi_2),
    \end{multline*}
    where by $J+L$ we mean $(j_1,j_2,\dots,j_s,\ell_1)$. Here, we used that $(a_1-a_2)$ is smooth and so the order of differentiation does not matter, and we can reorder and combine $\partial^L \partial^J=\partial^{J+L}$.

    \underline{Claim 2:} For any $\alpha=(\alpha_2,\dots,\alpha_n)$ we have
    \begin{equation}\label{eqn:claim2}
        \partial^\alpha (a_1-a_2)(0,0)=\partial^{\alpha}\tilde{q}(0,0).
    \end{equation}
    To see this, we compute
      \[
  \partial^{\alpha}\tilde{q}(x,\xib)=\partial^\alpha (a_1-a_2)(x_1,\xb(x_1),\xib(x_1)) \left( \frac{\partial \xi_2(x_1)}{\partial\xi_2}\right)^{\alpha_2} \left( \frac{\partial \xi_3(x_1)}{\partial\xi_3}\right)^{\alpha_3} \cdots  \left( \frac{\partial \xi_n(x_1)}{\partial\xi_n}\right)^{\alpha_n} +R_{\alpha}(x,\xib),
   \]
   where $R_{\alpha}$ is a sum of terms that have at least one factor of
   \[
   \partial^\beta x_j(x_1) \quad \text{for} \,\, |\beta|\leq |\alpha|, \qquad  \partial^\beta \xi_j(x_1) \quad \text{for} \,\, 2\leq |\beta|\leq |\alpha|, \qquad \text{or} \qquad \partial_{\xi_\ell}\xi_j(x_1) \quad \text{for} \,\, \ell \not=j
   \]
   for $j\in\{2,\dots,n\}.$  Moreover, since the flow satisfies
   \[
   \xb(x_1)=\xb+O(x_1) \quad \text{and} \quad \xib=\xib+O(x_1),
   \]
  we have $R_{\alpha}(0,\xb,\xib)=0$ and
   \[
   \frac{\partial \xi_j(x_1)}{\partial\xi_j}\Big|_{x_1=0}=1 \qquad \text{for all} \,\, j=2,\dots,n.
   \]
   Therefore $\partial^\alpha \tilde{q}(0,0)=\partial^\alpha (a_1-a_2)(0,0)$ as desired.

Now, we are ready to Taylor expand $\tilde{q}$ around $\xib=0$. Applying Taylor's Theorem with Remainder (see for example \cite[Chapter 4.2]{653Notes}) to $\tilde{q}(0,\xib)$ we have
\[
\tilde{q}(0,\xib)=\sum_{|J|\leq k+1} \frac{1}{|J|!}\partial^J \tilde{q}(0,0)\xib^J +R_{k+1}(\xib),
\]
where
\[
R_{k+1}(\xib)=\frac{1}{(k+1)!}\sum_{|J|=k+2}\left(\int_0^1 (1-s)^{k+1} \partial^J\tilde{q}(0,s\xib) ds \right)\xib^J.
\]
We can switch between the $J$ notation and the $\alpha$ notation the formula
\[
\sum_{|J|=s} \frac{1}{|J|!}\partial^J \tilde{q}(0,0)\xib^J =\sum_{|\alpha|=s} \frac{1}{\alpha!}\partial^\alpha \tilde{q}(0,0)\xib^\alpha. 
\]
Therefore, using (\ref{eqn:claim1}) and (\ref{eqn:claim2}) we have
\[
\tilde{q}(0,\xib)=\sum_{|J|\leq k+1} \frac{1}{|J|!}\partial^J (a_1-a_2)(0,0)\xib^J +R_{k+1}(\xib) =\sum_{|J|= k+1} \frac{1}{|J|!}\partial^J (a_1-a_2)(0,0)\xib^J +R_{k+1}(\xib).
\]
Next, let $q_0(\xib):=\sum_{|J|= k+1} \frac{1}{|J|!}\partial^J (a_1-a_2)(0,0)\xib^J$.  We will show that $q_0(\xib)\gtrsim |\xib|^{k+1}$. First, we make a few observations about $q_0$ and $R_{k+1}$. Since there is exactly one contact point on the support of $\chi$, we have that $a_1-a_2\geq 0$. In particular, it must be that $a_1-a_2\geq 0$ along each line of the form $\boldsymbol{0}^j\xi_j$ where by $\boldsymbol{0}^j$ we mean $(0,\dots,0,1,0,\dots,0)$, the $n-1$ tuple with $1$ in the $j$th slot an zero otherwise.
So for each $j\in\{2,\dots,n\}$ we have
\[
0\leq (a_1-a_2)(0,\boldsymbol{0}^j\xi_j)=\tilde{q}(\boldsymbol{0}^j\xi_j)=\frac{1}{(k+1)!}\partial_{\xi_j}^{k+1}(a_1-a_2)(0,0)\xi_j^{k+1}+R_{k+1}(\boldsymbol{0}^j\xi_j). 
\]
Here we see why we have assumed that $k$ is odd, as it is apparent from this inequality that $k+1$ must be even for there to be only one point of contact on the support of $\chi$. Next, we claim that for each $j$ it must be that 
\[
0 < \partial_{\xi_j}^{k+1}(a_1-a_2)(0,0). 
\]
If there were some $j$ where $ \partial_{\xi_j}^{k+1}(a_1-a_2)(0,0)=0$, then along the line $\boldsymbol{0}^j\xi_j$ we would have
\[
(a_1-a_2)(0,\boldsymbol{0}^j\xi_j)=R_{k+1}(\boldsymbol{0}^j\xi_j).
\]
However since $R_{k+1}$ is at least of order $k+2$ in $\xib$, along this line $a_1-a_2$ would at least exhibit $k+1$-th order contact, which contradicts our assumption that the contact is of $k$th order in each direction. 


Next we show that there exists a constant $C>0$ such that $q_0\geq C|\xib|^{k+1}$. To show this, we will use the contact condition. We begin by showing that for each $j\in\{2,\dots,n\}$ there exists $c_j>0$ such that $q_0(\xib)\geq c_j \xi_j^{k+1}$. First we show the inequality holds along each line $\bm^j\xi_j$ where recall $\bm^j=(m^{j}_2,\dots,m^{j}_{j-1},1,m^{j}_{j+1},\dots,m^{j}_n)$ for $m^{j}_i\in\R$. Along this line 
\[
\{\xi_2=m^{j}_2\xi_j, \dots, \xi_{j-1}=m^{j}_{j-1},\xi_j, \xi_{j+1}=m^{j}_{j+1}\xi_j,\dots, \xi_n=m^{j}_n\xi_j\}
\]
we have
\[
q_0(\bm^j\xi_j)=\sum_{|\alpha|=k+1}\frac{1}{\alpha!}\partial^{\alpha}(a_1-a_2)(0,0)(\bm^j)^{\alpha}\xi_j^{|\alpha|}=\xi_j^{k+1}\sum_{|\alpha|=k+1}\frac{1}{\alpha!}\partial^{\alpha}(a_1-a_2)(0,0)(\bm^j)^{\alpha}.
\]
There cannot exist an $\bm^j$ such that $\sum_{|\alpha|=k+1}\frac{1}{\alpha!}\partial^{\alpha}(a_1-a_2)(0,0)(\bm^j)^{\alpha}=0$ as that would imply that the order of contact would be larger than $k$th order along this line $\bm^j \xi_j$. Hence $\sum_{|\alpha|=k+1}\frac{1}{\alpha!}\partial^{\alpha}(a_1-a_2)(0,0)(\bm^j)^{\alpha}\not=0$ for all $\bm^j$. Furthermore, since
\[
\partial_{\xi_j}^{k+1}(a_1-a_2)(0,0)> 0
\]
it must be that $\sum_{|\alpha|=k+1}\frac{1}{\alpha!}\partial^{\alpha}(a_1-a_2)(0,0)(\bm^j)^{\alpha}$ is positive everywhere. Therefore, this polynomial in $\bm^j$ is bounded below by some positive constant. Let $c_j>0$ denote the minimum. Then we have 
\[
q_0(\bm^j\xi_j)\geq c_j \xi_j^{k+1},
\]
for each $j=2,\dots,n$ and $\bm^j$. We want to upgrade this to $q_0(\xib)\geq c_j \xi_j^{k+1}$ for any $\xib$. Note that for each $\ell\not=j$ we also have $q_0(\bm^\ell \xi_\ell)\geq c_\ell \xi_\ell^{k+1}\geq0$. Patching these estimates together we obtain
\[
q_0(\xib)\geq c_j \xi_j^{k+1} \quad \text{for each } \,\, j=2,\dots,n. 
\]
Moreover, this implies
\[
q_0(\xib)\geq \frac{1}{n-1}\sum_{j=2}^nc_j\xib^{k+1}\geq C |\xib|^{k+1},
\]
for some constant $C>0$. 

Finally, we have shown that
\[
\tilde{q}(0,\xib)=q_0(\xib)+R_{k+1}(\xib), \qquad \text{with} \,\, q_0(\xib)\geq C |\xib|^{k+1} \quad \text{and} \,\, |R_{k+1}(\xib)|\lesssim |\xib|^{k+2}.
\]
Therefore, in a small neighbourhood of the contact point, we can express
\[
\tilde{q}(x,\xib)=q_0(x,\xib)+R_{k+1}(x,\xib)
\]
where $q_0(x,\xib)\geq C(x)|\xib|^{k+1}$ and $C(0)>0$, and $|R_{k+1}(x,\xib)|\lesssim|\xib|^{k+2}$. Furthermore, we see that
\[
|\tilde{q}(x,\xib)|=|q_0(x,\xib)+R_{k+1}(x,\xib)|\geq |q_0(x,\xib)|-|R_{k+1}(x,\xib)|\geq C(x)|\xib|^{k+1}-\tilde{C}(x)|\xib|^{k+2},
\]
and so for $\xib$ sufficiently small (possibly shrinking the support of $\chi$ if needed) we can ensure that $\tilde{C}(x)|\xib|^{k+2}\leq \frac{1}{2}C(x)|\xib|^{k+1}$. This implies
\begin{equation*}
|\tilde{q}(x,\xib)|\geq \frac{C(x)}{2}|\xib|^{k+1}=c(x)|\xib|^{k+1},
\end{equation*}
where $c(0)>0$ as claimed.
  
\end{proof}


\section{Proof of Theorem \ref{thrm:linfcontact}}
This section presents the proof of Theorem \ref{thrm:linfcontact}.  To prove the $L^p$ estimates on $u_h$, we utilise that $u_h=W^*(x_1)v_h$ where $v_h$ is a quasimode for $hD_{x_1}$ and $\tilde{q}(x,hD_{\xb})$, and that $\tilde{q}(x,\xib)\gtrsim |\xib|^{k+1}$. Since $v_{h}$ is a quasimode for two such simple equations we are able to decompose $v_{h}$ very efficiently. As in \cite{TacyContact}, we will use a combination of wavelet and Fourier decompositions. The wavelet decomposition will take place in the $x_{1}$ variable. Wavelet decomposition is particularly suitable for our application. With a wavelet decomposition one writes a function of a superposition of dilations and translations of a ``mother-wavelet'' $f$. That is in terms of functions $f_{a,b}$ given by
\[f_{a,b}(t)=|a|^{-1/2}f\left(\frac{t-b}{a}\right)\quad a,b\in\R, \,\, a\not=0.\]
For our purposes, we will assume that the mother wavelet $f$ is smooth and has compactly supported Fourier transform. Moreover, we will have that
$$\int f(t)dt=0.$$
The wavelet transform
\[g \mapsto X_{g}(a,b)=\int |a|^{-1/2}f\left(\frac{t-b}{a}\right)g(t)dt\]
gives the ``wavelet coefficients'' in much the same way that the Fourier transform gives the Fourier coefficients for the expression of a function in terms of plane waves. Compared to the Fourier transform, the wavelet transform allows for both localisation in space and frequency, up to the limits placed by the uncertainty principle. Since $v_{h}$ obeys $hD_{x_{1}}v_{h}=O_{L^{2}}(h)$ we expect that $v_{h}$ is roughly constant in $x_{1}$. So, if $|a|\ll 1$, then we would expect that
\[X_{v_h(\cdot,\bar{x})}(a,b)=\int \frac{1}{|a|^{1/2}}f\left(\frac{x_{1}-b}{a}\right)v(x_{1},\bar{x})dx_{1}\approx C\int \frac{1}{|a|^{1/2}}f\left(\frac{x_{1}-b}{a}\right)dx_{1}=C|a|^{1/2}\int f(y) dy=0.\]
Hence the wavelet coefficients for $|a|\ll 1$ will be small. On the other hand, if $|a|\gg 1$, the fact that $v_{h}$ is localised to an $O(1)$ region will ensure that
\[ \int \frac{1}{|a|^{1/2}}f\left(\frac{x_{1}-b}{a}\right)v(x_{1},\bar{x})dx_{1}\approx |a|^{-1/2}\]
is small. Therefore the main contributions to $v_{h}$ come from the wavelets with $|a|\approx 1$ (so the decomposition is very sparse). The wavelet decomposition only uses the fact that $v_{h}$ is a quasimode of $hD_{x_{1}}$. To incorporate the fact that $v_{h}$ is also a quasimode of $\tilde{q}(x,hD_{\bar{x}})$ we take a Fourier transform in the $\bar{x}$ variables. From the ellipticity estimate
\begin{equation}\tilde{q}(x,\bar{\xi})\gtrsim|\xib|^{k+1}\label{eqn:elliptic}\end{equation}
we will find that the main contribution to $v_{h}$ comes from $\xib\approx 0$. Having obtained an efficient decomposition for $v_{h}$ we then apply $W^*(x_{1})$ to extract an expression for $u_{h}$. The $L^{p}$ estimates on $u_{h}$ follow from the composition of $W^*(x_{1})$ with the wavelet-frequency decomposition. We will abuse terminology somewhat and refer to the composition of $W^*(x_{1})$ with the wavelet-frequency decomposition of $v_{h}$ as a wavelet-frequency decomposition for $u_{h}$

We begin in Section \ref{sec:decomp} by producing such a wavelet-frequency decomposition for $u_{h}$.  Guided by the ellipticity estimate \ref{eqn:elliptic} and the wavelet heuristics discussed above will express the operator $W^{*}(x_1)$ in terms of operators $W_{a,j}^{*}(x_{1})$ where $W_{a,j}^{*}(x_{1})$ represents the contribution to $W^{*}(x_{1})$ from the wavelets $\{f_{a,b},b\in\R\}$ and frequencies $|\xib|\approx 2^{j}h^{\frac{1}{k+1}}$. Next, in Section \ref{sec:lpest}, we present two key lemmas, Lemma \ref{lem:ftterm} and Lemma \ref{lem:W*term}, needed to obtain the $L^p$ estimates on $u_h$. In Lemma \ref{lem:ftterm}, we control the mass of frequency localised wavelets, this allows us to see that the main contributions to $u_{h}$ come from the terms $W^{*}_{a,j}$ where $|a|\approx 1$ and $j\approx 0$. In Lemma \ref{lem:W*term}, we use a ``$TT^*$" argument to obtain estimates on the $W_{a,j}(x_1)$ operator. Assuming these lemmas, we prove Theorem \ref{thrm:linfcontact}. Lastly, in Section \ref{sec:lemmas}, we present the proofs of Lemma \ref{lem:ftterm} and Lemma \ref{lem:W*term}.

\subsection{A wavelet and frequency decomposition}\label{sec:decomp}
 We begin by exploiting the structure of $W^*$. Recall
\begin{align*}
u_h(x_1,\xb)=W^*(x_1)v_h(x_1,\yb) &=\frac{1}{(2\pi h)^{n-1}}\int e^{\frac{i}{h} \big(\varphi_1(x_1,\xb,\xib)-\langle \yb, \xib \rangle \big)}b^h(x,\xib)v_h(x_1,\yb)d\yb \, d\xib. 
\end{align*}
Using the Fourier transform, we have
\[
\Fh[v_h(x_1,\cdot)](\xib)=\frac{1}{(2\pi h)^{\frac{n-1}{2}}}\int e^{-\frac{i}{h}\langle \yb,\xib \rangle}v_h(x_1,\yb)d\yb,
\]
and so we can rewrite
\[
u_h(x_1,\xb)=\frac{1}{(2\pi h)^{\frac{n-1}{2}}} \int e^{\frac{i}{h} \varphi_1(x_1,\xb,\xib)}b^h(x,\xib) \Fh[v_h(x_1,\cdot)](\xib)  d\xib.
\]
Next, to use that $v_h$ is a quasimode of $hD_{x_1}$, we decompose $v_h$ into wavelets using a continuous wavelet transform in $x_1$. Let $f\in\cci(\R)$ denote the mother wavelet, which satisfies
\[
C_f:=\int \frac{|\hat{f}(\eta)|^2}{|\eta|}d\eta<\infty.
\]
Moreover, since $f\in\cci(\R)$, we know $\hat{f}$ is continuous, and using that $C_f<\infty$, we have that $\hat{f}(0)=0$. Thus $f$ must satisfy $\int f(t)dt=0$.
We can write $v_h$ as a superposition of wavelets $X_{v_h}$,
\[
v_h(x_1,\yb)=\frac{1}{C_f}\int \frac{1}{|a|^{5/2}}X_{v_h}(a,b,\yb) f\left(\frac{x_1-b}{a} \right) da \, db,
\]
where
\begin{equation}\label{eqn:X_vdef}
X_{v_h}(a,b,\yb)=\frac{1}{|a|^{1/2}}\int f\left( \frac{x_1-b}{a}\right)v_h(x_1,\yb) dx_1.
\end{equation}
Therefore,
\begin{align}
u_h(x)&=\frac{1}{C_f \, (2\pi h)^{\frac{n-1}{2}}}\int \frac{1}{|a|^{5/2}} e^{\frac{i}{h} \varphi_1(x_1,\xb,\xib)}b^h(x,\xib) f\left(\frac{x_1-b}{a} \right) \Fh[X_{v_h}(a,b,\cdot)](\xib) d\xib da db \nonumber \\
&=:\frac{1}{C_f} \int \frac{1}{|a|^{5/2}} W^{*}_a(x_1)\Fh[X_{v_h}(a,b,\cdot)](\xib) da,
\end{align}
where
\[
W^{*}_a(x_1)g(b,\xib):=\frac{1}{(2\pi h)^{\frac{n-1}{2}}}\int e^{\frac{i}{h} \varphi_1(x_1,\xb,\xib)}b^h(x,\xib) f\left(\frac{x_1-b}{a} \right) g(b,\xib) db \, d\xib.
\]
Next, we want to incorporate the property that $v_h$ is a quasimode for $\tilde{q}(x,hD_{\xb})$, where $\tilde{q}(x,\xib) \gtrsim |\xib|^{k+1}$. Now since $v_h$ satisfies
\[
\tilde{q}(x,hD_{\xb})v_h=O_{L^2}(h\|v_h\|_{L^2}), 
\]
we know that
\[
\|\tilde{q}(x,\xib)\Fh[v_h]\|_{L^2}\lesssim h\|\Fh[v_h]\|_{L^2}.
\]
Furthermore, using the lower bound on $\tilde{q}$ we have
\[
C(x)|\xib|^{k+1}\Fh[v_h]=O_{L^2}(h\|\Fh[v_h]\|_{L^2}).
\]
Therefore, we expect $\Fh[v_h]$ to be concentrated near $|\xib|\lesssim h^{\frac{1}{k+1}}$. To utilise this localisation in frequency, we will dyadically decompose $W^{*}_a(x_1)$ near $|\xib|\approx h^{\frac{1}{k+1}}$. To do so, let $\psi_0\in\cci(\R)$ be supported in $[-2,2]$ and $\psi\in\cci(\R)$ be supported $[1/2,3/2]$ such that on the support of $\chi$ we have
\[
1=\left(\psi_0(h^{-\frac{1}{k+1}}|\xib|)+\sum_{j=1}^J \psi(2^{-j} h^{-\frac{1}{k+1}} |\xib|) \right) 
\]
for $2^J h^{\frac{1}{k+1}}=1$. We denote 
\[
\psi_{j}(\xib)=\begin{cases} \psi_0(h^{-\frac{1}{k+1}}|\xib|) & j=0 \\ \psi(2^{-j} h^{-\frac{1}{k+1}} |\xib|) & j\geq 1
\end{cases}.
\]  
Then we define 
\begin{equation}\label{eqn:wajl}
W^{*}_{a,j}(x_1)g(b,\xib):=\frac{1}{2\pi h}\int e^{\frac{i}{h} \varphi_1(x_1,\xb,\xib)}b^h(x,\xib) f\left(\frac{x_1-b}{a} \right) \sqrt{\psi_{j}(\xib)} g(b,\xib) db \, d\xib,
\end{equation}
and write $u_h$ in terms of these operators as
\begin{equation}\label{eqn:udecomp}
u_h(x)=\sum_{j=0} \frac{1}{C_f} \int \frac{1}{|a|^{5/2}} W^{*}_{a,j}(x_1)\left[\sqrt{\psi_{j}(\xib)}\Fh[X_{v_h}(a,b,\cdot)](\xib)\right] da.
\end{equation}

\subsection{Obtaining the $L^p$ estimates}\label{sec:lpest} Let $p'$ satisfy $\frac{1}{p}+\frac{1}{p'}=1$.
Using the dual formulation of the $L^p$ norm, and the representation of $u_h$  from (\ref{eqn:udecomp}) we have,
\begin{align}
    \|u_h\|_{L^p}&=\sup_{\|G\|_{L^{p'}}\leq 1}\langle u_h(x_1,\xb),G(x_1,\xb) \rangle_{L^2_x} =\sup_{\|G\|_{L^{p'}}\leq 1} \int \langle u_h(x_1,\xb),G(x_1,\xb) \rangle_{L^2_{\xb}} dx_1 \nonumber\\
    &=\sup_{\|G\|_{L^{p'}}\leq 1} \int \left\langle \sum_{j=0} \frac{1}{C_f} \int \frac{1}{|a|^{5/2}} W^{*}_{a,j}(x_1)\left[\sqrt{\psi_{j}(\xib)}\Fh[X_{v_h}(a,b,\cdot)](\xib)\right] da, G(x_1,\xb) \right\rangle_{L^2_{\xb}} dx_1. \nonumber
    \end{align}
Interchanging the integrals and using the $W_{a,j}(x_1)$ operators we have
\begin{align}
    \|u_h\|_{L^p} &= \sup_{\|G\|_{L^{p'}}\leq 1} \frac{1}{C_f} \sum_{j=0} \int \frac{1}{|a|^{5/2}} \left\langle \sqrt{\psi_{j}(\xib)}\Fh[X_{v_h}(a,b,\cdot)](\xib) , \int  W_{a,j}(x_1)G(x_1,\xb) dx_1 \right\rangle_{L^2_{b,\xib}}  da\nonumber \\
    &\leq  \sup_{\|G\|_{L^{p'}}\leq 1} \frac{1}{C_f} \sum_{j=0} \int \frac{1}{|a|^{5/2}} \left\|  \sqrt{\psi_{j}(\xib)}\Fh[X_{v_h}(a,b,\cdot)](\xib) \right\|_{L^2_{b,\xib}} \left\| \int  W_{a,j}(x_1)G(x_1,\xb) dx_1\right\|_{L^2_{b,\xib}} da. \label{eqn:lpdual}
\end{align}
Therefore, to complete the proof of Theorem \ref{thrm:linfcontact} we first must obtain estimates on the $L^2$ norms of $\sqrt{\psi_{j}}\Fh[X_{v_h}]$ and $\int W_{a,j}(x_1)G dx_1$. Such estimates are presented in the following two lemmas.
\begin{lemma}\label{lem:ftterm}
    Let $X_{v_h}$ be as defined in \ref{eqn:X_vdef}, where $v_h$ is a strong joint quasimode of order $h$ for $hD_{x_1}$ and $\tilde{q}(x,hD_{\xb})$, and $\tilde{q}(x,\xib)$ satisfies (\ref{eqn:qstructure}).  Then, for any $M\in \N$
    \[
    \|\sqrt{\psi_{j}}\Fh[X_{v_h}] \|_{L^2_{\xib,b}} \lesssim \frac{\|v_h\|_{L^2_x}}{2^{j(k+1)M}} \times  \begin{cases} |a|^{3/2} & |a|\leq 1 \\1 & |a|\geq 1  \end{cases}.
    \]
    \end{lemma}

\begin{lemma}\label{lem:W*term}
    Let $\delta(n,p,k)$ and $W_{a,j}^*(x_1)$ be as defined in (\ref{eqn:delta}) and (\ref{eqn:wajl}) respectively. Let $\frac{1}{p}+\frac{1}{p'}=1$ and $2\leq p\leq \infty$. Then for $|a|\leq 1$ we have
       \[
 \left\| \int W_{a,j}(x_1)G(x_1,\xb) dx_1 \right\|_{L^2_{b,\xib}} \lesssim h^{-\delta(n,p,k)}|a|^{1/2}\|G\|_{L^{p'}_x} \begin{cases}
            1 & 2\leq p \leq \frac{2(n+1)}{n-1} \\
              2^{j\left( \frac{n-1}{2}-\frac{n+1}{p}\right)} & \frac{2(n+1)}{n-1}\leq p\leq \infty 
     \end{cases},
    \]
    and for $|a|\geq 1$ we have
            \[
   \left\| \int W_{a,j}(x_1)G(x_1,\xb) dx_1 \right\|_{L^2_{b,\xib}} \lesssim h^{-\delta(n,p,k)}|a|^{1/2} \|G\|_{L^{p'}_x} \begin{cases}
          |a|^{\frac{1}{p}-\frac{n-1}{2}\left(\frac{1}{2}-\frac{1}{p} \right)} & 2\leq p \leq \frac{2(n+1)}{n-1} \\ 
           2^{j\left( \frac{n-1}{2}-\frac{n+1}{p}\right)}  & \frac{2(n+1)}{n-1}\leq p\leq \infty 
     \end{cases}.
    \]
\end{lemma}

We first prove Theorem \ref{thrm:linfcontact} by combining the estimates from Lemma \ref{lem:ftterm} and Lemma \ref{lem:W*term}. We delay the proof of these lemmas for Section \ref{sec:lemmas}.

 \begin{proof}[Proof of Theorem \ref{thrm:linfcontact}]
We showed in (\ref{eqn:lpdual}) that we could control
\begin{equation}\label{eqn:upbdd}
\|u_h\|_{L^p} \leq \sup_{\|G\|_{L^{p'}}\leq 1} \frac{1}{C_f} \sum_{j=0} \int \frac{1}{|a|^{5/2}} \left\|  \sqrt{\psi_{j}(\xib)}\Fh[X_{v_h}(a,b,\cdot)](\xib) \right\|_{L^2_{b,\xib}} \left\| \int  W_{a,j}(x_1)G(x_1,\xb) dx_1\right\|_{L^2_{b,\xib}} da.
\end{equation}
We decompose the integral in $a$ into two pieces, $|a|\leq 1$ and $|a|\geq 1$. First, we consider the contributions from the interval $|a|\leq 1$. 
Applying Lemma \ref{lem:ftterm} and Lemma \ref{lem:W*term} for $|a|\leq 1$ we have for any $M\in\N$ that
\begin{align}
     \int_{|a|\leq 1} \frac{1}{|a|^{5/2}} & \left\|  \sqrt{\psi_{j}(\xib)}\Fh[X_{v_h}(a,b,\cdot)](\xib) \right\|_{L^2_{b,\xib}} \left\| \int  W_{a,j}(x_1)G(x_1,\xb) dx_1\right\|_{L^2_{b,\xib}} da \nonumber \\
     &\leq   \int_{|a|\leq 1} \frac{1}{|a|^{5/2}} \frac{\|v_h\|_{L^2}|a|^{3/2}}{2^{j(k+1)M}} |a|^{1/2} \|G\|_{L^{p'}} h^{-\delta(n,p,k)} 2^{\mu(j,n,p)} da \nonumber \\
     &= \|v_h\|_{L^2} h^{-\delta(n,p,k)}  2^{\mu(j,n,p)-j(k+1)M}  \|G\|_{L^{p'}}   \int_{|a|\leq 1} |a|^{-1/2} da \nonumber \\
     &\lesssim \|v_h\|_{L^2} h^{-\delta(n,p,k)}  2^{\mu(j,n,p)-j(k+1)M/2}  \|G\|_{L^{p'}}, 
 \label{eqn:thrma<1}
\end{align}
where $\delta(n,p,k)$ is as defined in (\ref{eqn:delta}) and
\[
\mu(j,n,p):=\begin{cases}
    0 & 2\leq p \leq \frac{2(n+1)}{n-1} \\
   j\left( \frac{n-1}{2}-\frac{n+1}{p} \right) & \frac{2(n+1)}{n-1} \leq p \leq \infty
\end{cases}.
\]
On the other hand, applying Lemma \ref{lem:ftterm} and Lemma \ref{lem:W*term} for $|a|\geq1,$ we have
\begin{align}
     \int_{|a|\geq1} \frac{1}{|a|^{5/2}} & \left\|  \sqrt{\psi_{j}(\xib)}\Fh[X_{v_h}(a,b,\cdot)](\xib) \right\|_{L^2_{b,\xib}} \left\| \int  W_{a,j}(x_1)G(x_1,\xb) dx_1\right\|_{L^2_{b,\xib}} da \nonumber \\
     &\leq  \frac{\|v_h\|_{L^2}}{2^{j(k+1)M}}  h^{-\delta(n,p,k)}  2^{\mu(j,n,p)} \|G\|_{L^{p'}} 
     \begin{cases}
         \int_{|a|\geq1} |a|^{-2+\frac{1}{p}-\frac{n-1}{2}\left(\frac{1}{2}-\frac{1}{p}\right)} da  & 2\leq p \leq \frac{2(n+1)}{n-1} \\ 
         \int_{|a|\geq1} |a|^{-2} da & \frac{2(n+1)}{n-1}\leq p \leq \infty
     \end{cases} \nonumber  \\
     &\lesssim  \|v_h\|_{L^2} h^{-\delta(n,p,k)} 2^{\mu(j,n,p)-j(k+1)M}    \|G\|_{L^{p'}} \label{eqn:thrma>1}.
\end{align}
Note that the integral with respect to $a$ in $(\ref{eqn:thrma>1})$ is convergent as
\[
\int_1^\infty a^{-2+\frac{1}{p}-\frac{n-1}{2}\left(\frac{1}{2}-\frac{1}{p}\right)}da \leq \int_1^\infty a^{-2+\frac{1}{p}}da <\infty
\]
since $p\geq2$.
Combining (\ref{eqn:thrma<1}) and (\ref{eqn:thrma>1}) with (\ref{eqn:upbdd}) we obtain for any $M\in\N$
\[
    \|u_h\|_{L^p}\lesssim  h^{-\delta(n,p,k)} \|v_h\|_{L^2_x}   \sum_{j=0} 2^{\mu(j,n,p)-j(k+1)M}. 
\]
Moreover, taking $M$ large, we see that the sum over $j$ is convergent and therefore
\[
\|u_h\|_{L^p}\lesssim  h^{-\delta(n,p,k)} \|v_h\|_{L^2_x}. 
\]
Lastly, using that $v_h=W(x_1)u_h$ and the unitarity of $W(x_1)$, we have
\[
\|u_h\|_{L^p}\lesssim  h^{-\delta(n,p,k)} \|u_h\|_{L^2}, 
\]
as claimed.
 \end{proof}

 \subsection{$L^2$ estimates and interpolation}\label{sec:lemmas}
 The $L^2$ estimates provided in Lemmas \ref{lem:ftterm} and \ref{lem:W*term} remain to be proved.  Simply using that $v_h$ is a strong joint quasimode for $hD_{x_1}$ and $\tilde{q}(x,hD_{\xb})$ we can prove the former as follows.  

 \begin{proof}[Proof of Lemma \ref{lem:ftterm}]
    We begin with the case where $|a|\geq 1$. We have two subcases to consider: (1) $j\geq 1$ and (2) $j=0$.
    
   If $j\geq 1$, then on the support of $\psi_{j}(\xib)$ we know $\tilde{q}(x,\xib)$ is supported away from zero and so we can write
   \[
  \sqrt{\psi_{j}} \Fh[X_{v_h}](\xib) =\frac{  \sqrt{\psi_{j}} }{\tilde{q}^M(x,\xib)} \tilde{q}^M(x,\xib)  \Fh[X_{v_h}](\xib).
   \]
   Therefore, 
   \begin{align*}
   \|  \sqrt{\psi_{j}} \Fh[X_{v_h}](\xib) \|_{L^2_{\xib}} & \leq \sup_{\xib} \left|  \frac{  \sqrt{\psi_{j}} }{\tilde{q}^M(x,\xib)}  \right| \| \tilde{q}^M(x,\xib) \Fh[X_{v_h}](\xib) \|_{L^2_{\xib}}.
   \end{align*}
   However, using the ellipticity of $\tilde{q}$ from (\ref{eqn:qstructure}), we have
   \[
  \left|  \frac{  \sqrt{\psi_{j}} }{\tilde{q}^M(x,\xib)}  \right| \lesssim \left|  \frac{  \sqrt{\psi_{j} }}{ \left( C(x)|\xib|^{k+1} \right)^M}  \right|.
   \]
  Thus, since $\psi_{j}$ is supported for $\frac{1}{2} h^{\frac{1}{k+1}}2^{j}\leq |\xib| \leq \frac{3}{2} h^{\frac{1}{k+1}}2^{j} $, we can control
\begin{equation}\label{eqn:a1j1}
  \|  \sqrt{\psi_{j}} \Fh[X_{v_h}](\xib) \|_{L^2_{\xib}}  \lesssim \frac{1}{h^M 2^{j(k+1)M}  } \| \tilde{q}^M(x,\xib) \Fh[X_{v_h}](\xib) \|_{L^2_{\xib}} \lesssim \frac{1}{h^M 2^{j(k+1)M}  }  \| \tilde{q}^M(x,hD_{\xb})  X_{v_h} \|_{L^2_{\xb}}.
  \end{equation}
   Recall from (\ref{eqn:X_vdef}) that
   \[
   X_{v_h}(a,b,\xb)=\frac{1}{|a|^{1/2}}\int f\left(\frac{x_1-b}{a} \right)v_h(x_1,\xb)dx_1.
   \]
   Moreover, since $v_h$ is localised in an $O(1)$ region we have
   \[
   | \tilde{q}^M(x,hD_{\xb}) X_{v_h}(a,b,\xb)| \leq \frac{1}{|a|^{1/2}} \int \left| f \left(\frac{x_1-b}{b} \right) \tilde{q}^M(x,hD_{\xb})v_h(x_1,\xb) \right| dx_1 \lesssim \frac{1}{|a|^{1/2}}\|  \tilde{q}^M(x,hD_{\xb})v_h\|_{L^2_{x_1}}.
   \]
This, combined with $(\ref{eqn:a1j1})$ implies
   \[
    \|  \sqrt{\psi_{j}} \Fh[X_{v_h}](\xib) \|_{L^2_{\xib}} \lesssim \frac{ \| \tilde{q}^M(x,hD_{\xb})v_h \|_{L^2_{x}}}{|a|^{1/2} h^M 2^{j(k+1)M}}   \lesssim \frac{ \| v_h \|_{L^2_{x}}}{|a|^{1/2}2^{j(k+1)M}} ,
   \]
   where we used that $v_h$ is a strong quasimode of order $h$ for $\tilde{q}(x,hD_{\xb})$. Lastly, note that if $b \gg a$ the support of $f((x_1-b)/a)$ does not over lap with the support of $v_h$, and so $X_{v_h}(a,b,\xb)=0$ for $b \gg a$. Therefore
   \[
   \|  \sqrt{\psi_{j}} \Fh[X_{v_h}](\xib) \|_{L^2_{\xib,b}} \lesssim \left(\int_{-a}^a \|  \sqrt{\psi_{j}} \Fh[X_{v_h}](\xib) \|_{L^2_{\xib}}^2   db\right)^{1/2} \lesssim \frac{|a|^{1/2}}{|a|^{1/2} 2^{j(k+1)M}} \| v_h \|_{L^2_{x}}=\frac{\|v_h\|_{L^2_x}}{2^{j(k+1)M}} 
   \]
   as desired. 
   
   Next, we consider the case where $j=0$. Here we note
       \[
       \|\sqrt{\psi_{0}} \Fh[X_{v_h}](\xib)\|_{L^2_{\xib,b}} \lesssim  \|  \Fh[X_{v_h}]  \|_{L^2_{\xib,b}} =  \| X_{v_h}  \|_{L^2_{\xb,b}}. 
       \]
       Furthermore, since $v_h$ is supported in an $O(1)$ neighbourhood and using the definition of $X_{v_h}$ (\ref{eqn:X_vdef}) we know that $|X_{v_h}|\lesssim |a|^{-1/2}\|v_h\|_{L^2_{x_1}}$. Hence $\|X_{v_h}\|_{L^2_{\xb}}\lesssim|a|^{-1/2}\|X_{v_h}\|_{L^2_x}$. Once again, using the support properties of $v_h$ and $f$ we have
       \[
       \|X_{v_h} \|_{L^2_{\xb,b}} \lesssim \left(\int_{-a}^a \|X_{v_h}  \|_{L^2_{\xb}} db\right)^{1/2} \lesssim\frac{|a|^{1/2}}{|a|^{1/2}} \|v_h\|_{L^2_x}. 
       \]
       Thus we have
       \[
        \|\sqrt{\psi_{0}} \Fh[X_{v_h}](\xib)\|_{L^2_{\xib,b}} \lesssim \|v_h\|_{L^2} = \frac{1}{2^{0(k+1)M}}\|v_h\|_{L^2}.
       \]
       This proves the claim when $|a|\geq 1$.

       Next, we consider the case where $|a|\leq 1$. Here, we note that since the mother wavelet satisfies $\int f(t) dt=0$, there exists a function $g$ such that $f(t)=i D_t g(t)$. Thus
    \[
    f\left(\frac{x_1-b}{a}\right)=a i D_{x_1} g\left( \frac{x_1-b}{a}\right). 
    \]
    Writing $X_{v_h}$ in terms of $g$ we have
    \begin{align}
    X_{v_h}(a,b,\xb)=\frac{1}{|a|^{1/2}} \int f\left(\frac{x_1-b}{a} \right) \overline{v_h(x_1,\xb)} dx_1 &=\frac{a}{|a|^{1/2}} \int i D_{x_1} g\left(\frac{x_1-b}{a} \right) \overline{v_h(x_1,\xb)} dx_1 \nonumber \\
    &=\frac{-ia}{h|a|^{1/2}} \int  g\left(\frac{x_1-b}{a} \right) \overline{ (hD_{x_1}) v_h(x_1,\xb) } dx_1. \label{eqn:Xwg}
    \end{align}
    We have two sub-cases to consider: (1) $j\geq 1$ and (2) $j=0$. First, suppose $j\geq 1$. Then just as in the $|a|\geq 1$ case we have
    \[
     \|  \sqrt{\psi_{j}} \Fh[X_{v_h}](\xib) \|_{L^2_{\xib}} \lesssim \frac{ \| \tilde{q}^M(x,hD_{\xb})  X_{v_h} \|_{L^2_{\xb}}}{h^M 2^{j(k+1)M}},
    \]
    where we used that $\tilde{q}$ is bounded away from zero on the support of $\psi_{j}$. Next, since $X_{v_h}$ satisfies (\ref{eqn:Xwg}),
    \[
    |\tilde{q}^M(x,hD_{\xb})X_{v_h}| \lesssim \frac{|a|}{h|a|^{1/2}}\int \left|g\left(\frac{x_1-b}{b}\right)\right| \left|\tilde{q}^M(x,hD_{\xb}) (hD_{x_1}) v_h \right| dx_1 \lesssim \frac{|a|}{h}\|\tilde{q}^M(x,hD_{\xb})(hD_{x_1})v_h\|_{L^2_{x_1}},
    \]
     where we used that $\|g((x_1-b)/a)\|_{L^2_{x_1}}\lesssim a^{1/2}$. Therefore we see that
    \[
     \|  \sqrt{\psi_{j}} \Fh[X_{v_h}](\xib) \|_{L^2_{\xib}} \lesssim \frac{|a| \|\tilde{q}^M(x,hD_{\xb}) (hD_{x_1})v_h\|_{L^2_x}}{h h^M 2^{j(k+1)M}}.
    \]
    Moreover, since $v_h$ is a strong joint quasimode of $\tilde{q}(x,hD_{\xb})$ and $hD_{x_1}$, and using that $X_{v_h}$ is zero for $b\gg a$ we have
    \[
     \|  \sqrt{\psi_{j}} \Fh[X_{v_h}](\xib) \|_{L^2_{\xib,b}} \lesssim \frac{|a|^{3/2}}{2^{j(k+1)M}} \| v_h\|_{L^2_x}.
    \]
    
    Lastly, for $j=0$, note
    \[
     \|\sqrt{\psi_{0}} \Fh[X_{v_h}](\xib)\|_{L^2_{\xib,b}} \lesssim  \|  \Fh[X_{v_h}]  \|_{L^2_{\xib,b}} =  \| X_{v_h}  \|_{L^2_{\xb,b}}. 
     \]
     Moreover, since $X_{v_h}$ satisfies (\ref{eqn:Xwg}) we have 
     \[
     |X_{v_h}|\lesssim \frac{|a|^{1/2}}{h} \int \left|g\left(\frac{x_1-b}{a} \right)\right| |hD_{x_1}v_h|dx_1\lesssim \frac{|a|}{h}\|hD_{x_1}v_h\|_{L^2_{x_1}},
     \]
     which implies $\|X_{v_h}\|_{L^2_{\xb}}\lesssim h^{-1}|a| \|hD_{x_1}v_h\|_{L^2_x}$. Lastly, integrating over $b$ and using that $v_h$ is a quasimode of $hD_{x_1}$ we obtain
     \[
    \| X_{v_h}  \|_{L^2_{\xb,b}} \lesssim \frac{|a|^{3/2}}{h} \|(hD_{x_1})v_h\|_{L^2_x} \lesssim |a|^{3/2}\|v_h\|_{L^2_x}.
     \]
     Combining, we obtain
     \[
     \|\sqrt{\psi_{0}} \Fh[X_{v_h}](\xib)\|_{L^2_{\xib,b}} \lesssim \frac{|a|^{3/2}}{2^{0(k+1)M}}\|v_h\|_{L^2_x}, 
     \]
     as desired.    
 \end{proof}

 Lastly, we must prove Lemma \ref{lem:W*term}. To do so, we utilise a ``$TT^*$" argument and Riesz-Thorin Interpolation. 

 \begin{proof}[Proof of Lemma \ref{lem:W*term}]
     We begin by considering the square of the $L^2$-norm of $\int W_{a,j}(x_1)G dx_1$. We compute
     \begin{align}
         &\left\| \int W_{a,j}(x_1)G(x_1,\xb) dx_1 \right\|^2_{L^2_{b,\xib}} \nonumber \\
         &\qquad =\iint \left\langle W_{a,j}(x_1)G(x_1,\xb), W_{a,j}(z_1)G(z_1,\xb) \right\rangle_{L^2_{b,\xib}} dx_1 \, dz_1 \nonumber \\
         &\qquad = \iint  \left\langle G(x_1,\xb),  W^{*}_{a,j}(x_1) W_{a,j}(z_1)G(z_1,\cdot) \right\rangle_{L^2_{\xb}} dx_1 \, dz_1  \nonumber \\
         &\qquad \leq \iint \|G(x_1,\xb)\|_{L^{p'}_{\xb}} \|W^{*}_{a,j}(x_1) W_{a,j}(z_1) \|_{L^{p'}_{\xb}\to L^{p}_{\xb}}  \|G(z_1,\xb)\|_{L^{p'}_{\xb}} dx_1 \, dz_1. \label{eqn:goal}
     \end{align}
     Therefore, we need an $L^{p'}\to L^p$ estimate on $W^{*}_{a,j}(x_1) W_{a,j}(z_1).$ To obtain such an estimate we first find $L^1 \to L^\infty$ and $L^2\to L^2$ estimates, then we interpolate.

     Beginning with the $L^1\to L^\infty$ estimate, we compute $W^{*}_{a,j}(x_1) W_{a,j}(z_1)$,
     \[
     W^{*}_{a,j}(x_1) W_{a,j}(z_1) G(\bar{z})=\int K_{a,j}(x_1,z_1,\bar{z},\xb) G(\bar{z}) d \bar{z},
     \]
     where
     \[
     K_{a,j}(x_1,z_1,\bar{z},\xb)=\frac{1}{(2\pi h)^{n-1}}\int e^{\frac{i}{h} \left( \varphi_1(x_1,\xb,\xib)-\varphi_1(z_1,\bar{z},\xib) \right)} b^h(x,\xib) b^h(z,\xib) f\left( \frac{x_1-b}{a}\right) f\left( \frac{z_1-b}{a}\right) \psi_{j}(\xib) d\xib \, db.
     \]
     Now, if $|x_1-z_1|\gg a,$ then the support of $f((x_1-b)/a)$ and $f((z_1-b)/a)$ will not overlap and hence the kernel is zero. So we only need to consider the case where $|x_1-z_1|\lesssim a$. To obtain the $L^1\to L^\infty$ estimate, we must bound the kernel $K_{a,j}$. First, note that we can control the integral in $b$ using the support properties of $f$, by
     \begin{equation}\label{eqn:bint}
    \left| \int f\left( \frac{x_1-b}{a}\right) f\left( \frac{z_1-b}{a}\right) db \right| \leq \left\|f\left( \frac{x_1-b}{a}\right)  \right\|_{L^2_b} \left\|f\left( \frac{z_1-b}{a}\right)  \right\|_{L^2_b} \lesssim |a|.
     \end{equation}
     To estimate the integral in $\xib$, we could just use the support properties of $\psi_{j}$. Combining this with (\ref{eqn:bint}) would give
     \begin{equation}\label{eqn:trivial}
  |K_{a,j}|\lesssim \frac{1}{(2\pi h)^{n-1}}|a| h^{\frac{n-1}{k+1}}2^{j(n-1)} \lesssim |a| h^{-(n-1)\left(1-\frac{1}{k+1} \right)}2^{j(n-1)}. 
  \end{equation}
  On the other hand, cancellation due to oscillations often allows us to obtain better estimates via the Stationary Phase Lemma. 
 We need to examine the phase function further to evaluate the integral in $\xib$. 
Since $\varphi_1$ satisfies (\ref{eqn:phi1def}), we can write
     \begin{align*}
          \Phi(x_1, \xb,z_1,\bar{z},\xib)&:=  \varphi_1(x_1,\xb,\xib)-\varphi_1(z_1,\bar{z},\xib)=\varphi_1(x_1,\xb,\xib)-\varphi_1(x_1,\bar{z},\xib)+\varphi_1(x_1,\bar{z},\xib)-\varphi_1(z_1,\bar{z},\xib) \\
            &=\langle \xb-\bar{z},\xib \rangle +\langle \xb-\bar{z},x_1 F(x_1,\xb,\bar{z},\xib) \rangle+(x_1-z_1) a_1(0,\bar{z},\xib)+O(|x_1-z_1|^2)
     \end{align*}
     for some $F\in C^{\infty}(\R^{1+3(n-1)})$.
     Therefore,
     \[
     \partial_{\xib} \Phi=(x_1 \partial_{\xib}F+1)(\xb-\bar{z})+(x_1-z_1)\partial_{\xib}a_1(0,\bar{z},\xib) +O(|x_1-z_1|^2)
     \]
     and 
     \[
     \partial_{\xib}^2\Phi=x_1 \partial_{\xib}^2 F (\xb-\bar{z})+(x_1-z_1)\partial_{\xib}^2 a_1(0,\bar{z},\xib)+O(|x_1-z_1|^2).
     \]
    The existence of a stationary point implies
     \[
     \xb-\bar{z}=O(x_1-z_1),
     \]
     since for $x_1$ sufficiently small we have $(x_1 \partial_{\xib}F+1)\not=0$. Therefore, near the stationary point, we have
     \begin{equation}\label{eqn:hess}
    \partial_{\xib}\Phi=O(|x_1-z_1|) \qquad \text{and} \qquad \partial_{\xib}^2\Phi=(x_1-z_1) \left(\partial_{\xib}^2 a_1(0,\bar{z},\xib)+O(|x_1-z_1|+|x_1|) \right).
     \end{equation}
    The curvature condition on $p_1$ implies that $|\partial_{\xib}^2a_1|>c>0$ and hence, near the stationary point we have
    \[
     |\det \partial_{\xib}^2\Phi|\geq c|x_1-z_1|^{n-1}.
    \]
    The Stationary Phase Lemma would imply that
    \[
 \left|  \frac{1}{(2\pi h)^{n-1}}\int e^{\frac{i}{h} \left( \varphi_1(x_1,\xb,\xib)-\varphi_1(z_1,\bar{z},\xib) \right)} b^h(x,\xib) b^h(z,\xib) \psi_{j}(\xib) d\xib  \right| \lesssim h^{-(n-1)} \left(\frac{h}{|x_1-z_1|}  \right)^{\frac{n-1}{2}}.
    \]
    Unfortunately, the Stationary Phase Lemma cannot be directly applied since the cutoff $\psi_{j}$ is $h$-dependent. However, for our purposes, it is enough to have the estimate form of the Stationary Phase Lemma as we do not need the full asymptotic expansion. Therefore we can use the Van der Corput Lemma (see for example \cite[Lemma 1.1.2]{SoggeBlue} for $d=1$ case, and higher dimensions look to \cite[Theorem 0.2]{TacySP} or \cite[Theorem 6]{SPTypeEst}). The proof of the Van der Corput Lemma proceeds by first placing a cut off function around the critical point of scale $s_{c}=h^{\frac{1}{2}}|\operatorname{Hess}|^{-\frac{1}{2d}}$, where $|\operatorname{Hess}|$ is the determinant of the Hessian and $d$ the dimension of the oscillatory integral. At a distance greater than $s_{c}$ from the critical point the non-degeneracy assumption ensures that the oscillation from the complex exponential is enough to induce decay (via an integration by parts argument). In the region within $s_{c}$ of the critical point, the complex exponential does not oscillate enough to give significant cancellation, and the integral over this region is estimated by support properties only. This argument still works when the symbol depends on $h$, so long as the regularity loss per derivative is less than $s_{c}^{-1}$ (for completeness we include a proof of this in Appendix \ref{sec:app}).  
    Therefore, since we lose $h^{-\frac{1}{k+1}}2^{-j}$ each time we differentiate $\psi_j(|\xib|)$, we need
    \[
    h^{-\frac{1}{k+1}}2^{-j}\leq s_c^{-1}=h^{-\frac{1}{2}}|\operatorname{Hess}|^{\frac{1}{2(n-1)}}\lesssim h^{-\frac{1}{2}}|x_1-z_1|^{\frac{1}{2}},
    \]
    where we use the structure of the Hessian from (\ref{eqn:hess}). So, in the region where $|x_1-z_1|\geq h^{1-\frac{2}{k+1}}2^{-2j}$ we can apply Van der Corput. When $|x_1-z_1|\leq h^{1-\frac{2}{k+1}}2^{-2j}$ we should just accept the trivial estimate (\ref{eqn:trivial}). Therefore, we have
    \begin{equation}\label{eqn:l1linf}
         \|  W^{*}_{a,j}(x_1) W_{a,j}(z_1) \|_{L^1_{\xb}\to L^\infty_{\xb}} \lesssim \begin{cases}
      |a| h^{-\frac{n-1}{2}}|x_1-z_1|^{-\frac{n-1}{2}} &\text{when }\, \, |x_1-z_1|\geq h^{1-\frac{2}{k+1}}2^{-2j} \\ |a| h^{-(n-1)\left(1-\frac{1}{k+1} \right)}2^{j(n-1)} &\text{when }\, \, |x_1-z_1|\leq h^{1-\frac{2}{k+1}}2^{-2j}
  \end{cases}.
    \end{equation}
     
 Next, we seek a $L^2\to L^2$ estimate. First, we note 
 \begin{align*}
     W^{*}_a(x_1)g(b,\xib) &=\frac{1}{(2\pi h)^{\frac{n-1}{2}}}\int e^{\frac{i}{h} \varphi_1(x_1,\xb,\xib)}b^h(x,\xib) f\left(\frac{x_1-b}{a} \right) g(b,\xib) db \, d\xib \\
     &=\frac{1}{(2\pi h)^{n-1}}\int e^{\frac{i}{h} \left( \varphi_1(x_1,\xb,\xib) -\yb\cdot \xib \right)}b^h(x,\xib) f\left(\frac{x_1-b}{a} \right) \Fh^{-1}[g(b,\cdot)](\yb) db \,d\yb \, d\xib \\
     &=W^*(x_1)\left[ \int f\left(\frac{x_1-b}{a} \right) \Fh^{-1}[g(b,\cdot)](\yb) db  \right].
 \end{align*}
 Furthermore, since $W^*$ is unitary, we see that
 \[
 \|W^{*}_a(x_1)g(b,\xib)\|_{L^2_{\xb}}\lesssim \left\| \int f\left(\frac{x_1-b}{a} \right) \Fh^{-1}[g(b,\cdot)](\yb) db  \right\|_{L^2_{\yb}} \leq \left\| f\left(\frac{x_1-b}{a} \right)\right\|_{L^2_b} \|g(b,\xib) \|_{L^2_{b,\xib}}. 
 \]
 Thus, since $\|f((x_1-b)/a)\|_{L^2_b}\lesssim |a|^{1/2}$, we see that $\|W^{*}_a(x_1)\|_{L^2_{b,\xib}\to L^2_{\xb}}\lesssim |a|^{1/2}$, and so
 \begin{equation}\label{eqn:l2l2}
  \|  W^{*}_{a,j}(x_1) W_{a,j}(z_1) \|_{L^2_{\xb}\to L^2_{\xb}} \lesssim \|  W_{a}^{*}(x_1) W_{a}(z_1) \|_{L^2_{\xb}\to L^2_{\xb}} \lesssim |a|.
 \end{equation}
Interpolating between the $L^1\to L^\infty$ (\ref{eqn:l1linf}) and $L^2\to L^2$ (\ref{eqn:l2l2}) estimates, we obtain
 \begin{equation}\label{eqn:LpLP}
  \|  W^{*}_{a,j}(x_1) W_{a,j}(z_1) \|_{L^{p'}_{\xb}\to L^{p}_{\xb}} \lesssim 
  \begin{cases}
      |a| h^{-\frac{n-1}{2}\left(1-\frac{2}{p} \right)}|x_1-z_1|^{-\frac{n-1}{2}\left(1-\frac{2}{p} \right)} &\text{when }\, \, |x_1-z_1|\geq h^{1-\frac{2}{k+1}}2^{-2j} \\ 
      |a| h^{-(n-1)\left(1-\frac{1}{k+1} \right)\left(1-\frac{2}{p} \right)}2^{j(n-1)\left(1-\frac{2}{p} \right)} &\text{when }\, \, |x_1-z_1|\leq h^{1-\frac{2}{k+1}}2^{-2j}
  \end{cases}.
 \end{equation}
 So, as we saw in (\ref{eqn:goal}), in order to control the square of the $L^2$-norm of $\int W_{a,j}(x_1)G dx_1$ we need to bound
 \[
  \iint \|G(x_1,\xb)\|_{L^{p'}_{\xb}} \|W^{*}_{a,j}(x_1) W_{a,j}(z_1) \|_{L^{p'}_{\xb}\to L^{p}_{\xb}}  \|G(z_1,\xb)\|_{L^{p'}_{\xb}} dx_1 \, dz_1.
 \]
 Using the two cases of our $L^{p'}\to L^p$ estimate (\ref{eqn:LpLP}), we find that we must bound both
  \begin{equation}\label{eqn:geq}
|a|h^{-\frac{n-1}{2}\left(1-\frac{2}{p} \right)}\iint \|G(x_1,\xb)\|_{L^{p'}_{\xb}}   \|G(z_1,\xb)\|_{L^{p'}_{\xb}} |x_1-z_1|^{-\frac{n-1}{2}\left(1-\frac{2}{p} \right)} dx_1 \, dz_1 \quad \text{when} \,\, |x_1-z_1|\geq h^{1-\frac{2}{k+1}}2^{-2j}
 \end{equation}
 and
 \begin{equation}\label{eqn:leq}
|a| h^{-(n-1)\left(1-\frac{1}{k+1} \right)\left(1-\frac{2}{p} \right)}2^{j(n-1)\left(1-\frac{2}{p} \right)} \iint \|G(x_1,\xb)\|_{L^{p'}_{\xb}}   \|G(z_1,\xb)\|_{L^{p'}_{\xb}}   dx_1 \, dz_1 \quad \text{when} \,\, |x_1-z_1|\leq h^{1-\frac{2}{k+1}}2^{-2j}. 
 \end{equation}
 We begin with controlling (\ref{eqn:leq}). We compute
\begin{align}
    \iint & \|G(x_1,\xb)\|_{L^{p'}_{\xb}}   \|G(z_1,\xb)\|_{L^{p'}_{\xb}} \mathds{1}_{\left\{|x_1-z_1|\leq h^{1-\frac{2}{k+1}}2^{-2j} \right\}}  dx_1 \, dz_1  \nonumber \\
    &= \int \|G(x_1,\xb)\|_{L^{p'}_{\xb}}   \left( \|G(x_1,\xb)\|_{L^{p'}_{\xb}} * \mathds{1}_{\left\{|x_1|\leq   h^{1-\frac{2}{k+1}}2^{-2j}\right\}}  \right) dx_1 \nonumber \\
    & \leq \left\| \|G(x_1,\xb)\|_{L^{p'}_{\xb}}  \right\|_{L^{p'}_{x_1}} \left\|  \|G(x_1,\xb)\|_{L^{p'}_{\xb}} * \mathds{1}_{\left\{|x_1|\leq   h^{1-\frac{2}{k+1}}2^{-2j} \right\}} \right\|_{L^{p}_{x_1}}. \label{eqn:Young} 
\end{align}
Furthermore, using Young's inequality on the second term, we can bound (\ref{eqn:Young}) by
\[
   \|G\|_{L^{p'}_x}^2  \left\|  \mathds{1}_{\left\{|x_1|\leq  h^{1-\frac{2}{k+1}}2^{-2j} \right\}} \right\|_{L^{p/2}_{x_1}} \lesssim \|G\|_{L^{p'}_x}^2 h^{\frac{2}{p}-\frac{4}{p(k+1)}}2^{-4j/p}.
\]
    Therefore  (\ref{eqn:leq}) is bounded by
    \begin{multline}\label{eqn:bddleq}
  |a| h^{-(n-1)\left(1-\frac{1}{k+1} \right)\left(1-\frac{2}{p} \right)}2^{j(n-1)\left(1-\frac{2}{p} \right)} \|G\|_{L^{p'}_x}^2 h^{\frac{2}{p}-\frac{4}{p(k+1)}}2^{-4j/p}  \\= |a| h^{-(n-1) +\frac{2n}{p}+\frac{1}{k+1}\left(n-1-\frac{2(n+1)}{p} \right)} 2^{j\left(n-1-\frac{2(n+1)}{p} \right)} \|G\|_{L^{p'}_x}^2.
    \end{multline}
Next, we consider (\ref{eqn:geq}). 
    In a similar manner to the previous calculation, we compute
    \begin{align*}
        \iint & \|G(x_1,\xb)\|_{L^{p'}_{\xb}}   \|G(z_1,\xb)\|_{L^{p'}_{\xb}}|x_1-z_1|^{-\frac{n-1}{2}\left(1-\frac{2}{p} \right)} \mathds{1}_{\left\{ h^{1-\frac{2}{k+1}} 2^{-2j} \leq  |x_1-z_1|\leq a \right\}} dx_1 \, dz_1 \\
        & = \int \|G(x_1,\xb)\|_{L^{p'}_{\xb}}  \left( \|G(x_1,\xb)\|_{L^{p'}_{\xb}}  * |x_1|^{-\frac{n-1}{2}\left(1-\frac{2}{p} \right)} \mathds{1}_{\left\{ h^{1-\frac{2}{k+1}} 2^{-2j} \leq  |x_1|\leq a \right\}} \right) dx_1 \\
        & \leq \|G(x_1,\xb)\|_{L^{p'}_x}^2 \left( \int_{h^{1-\frac{2}{k+1}}2^{-2j}}^a |x_1|^{-\frac{n-1}{2}\left(1-\frac{2}{p} \right)\frac{p}{2}} dx_1 \right)^{2/p},
    \end{align*}
    where we used Young's Inequality in the third line.
    Furthermore
    \[
    \left( \int_{h^{1-\frac{2}{k+1}}2^{-2j}}^a |x_1|^{-\frac{n-1}{2}\left(1-\frac{2}{p} \right)\frac{p}{2}} dx_1 \right)^{2/p}\lesssim \begin{cases}
        |a|^{-\frac{n-1}{2}+\frac{n+1}{p}} & 2\leq p <\frac{2(n+1)}{n-1} \\
        h^{\left( -\frac{n-1}{2}+\frac{n+1}{p} \right)\left(1-\frac{2}{k+1} \right)} 2^{-2j \left( -\frac{n-1}{2}+\frac{n+1}{p} \right)}& \frac{2(n+1)}{n-1} <p\leq \infty
    \end{cases}.
    \]
    Note that if $2\leq p<\frac{2(n+1)}{n-1}$ then the exponent $-\frac{n-1}{2}+\frac{n+1}{p}$ is positive and hence 
    \[
    |a|^{-\frac{n-1}{2}+\frac{n+1}{p}} \leq \begin{cases}
        1 & |a|\leq 1  \\ |a|^{-\frac{n-1}{2}+\frac{n+1}{p}}  & |a|\geq 1
    \end{cases} \qquad \text{for}\,\, 2\leq p<\frac{2(n+1)}{n-1}.
    \]
   We use Hardy-Littlewood-Sobolev to resolve the edge case where $p=\frac{2(n+1)}{n-1}$.
Therefore we can bound (\ref{eqn:geq}) by
    \begin{align}
        |a|& h^{-\frac{n-1}{2}\left(1-\frac{2}{p} \right)}  \|G\|_{L^{p'}_x}^2 \times \begin{cases} 1 & |a|\leq 1 \,\, \text{and} \,\, 2\leq p \leq \frac{2(n+1)}{n-1} \\  
          |a|^{-\frac{n-1}{2}+\frac{n+1}{p}}  &|a|\geq 1 \,\, \text{and} \,\, 2\leq p \leq \frac{2(n+1)}{n-1} \\  
       h^{\left( -\frac{n-1}{2}+\frac{n+1}{p} \right)\left(1-\frac{2}{k+1} \right)} 2^{-2j \left( -\frac{n-1}{2}+\frac{n+1}{p} \right)} & \frac{2(n+1)}{n-1} \leq p\leq \infty \end{cases}. \label{eqn:bddgeq}
    \end{align}
    Note we can expand the exponent to obtain
    \[
     h^{-\frac{n-1}{2}\left(1-\frac{2}{p} \right)}   h^{\left( -\frac{n-1}{2}+\frac{n+1}{p} \right)\left(1-\frac{2}{k+1} \right)} 2^{-2j \left( -\frac{n-1}{2}+\frac{n+1}{p} \right)}=h^{-(n-1) +\frac{2n}{p}+\frac{1}{k+1}\left(n-1-\frac{2(n+1)}{p} \right)} 2^{j\left(n-1-\frac{2(n+1)}{p} \right)}. 
    \]
    Combining (\ref{eqn:bddleq}) and (\ref{eqn:bddgeq}) with (\ref{eqn:goal}) we see that we can control
    \[
    \left\| \int W_{a,j}(x_1)G(x_1,\xb) dx_1 \right\|^2_{L^2_{b,\xib}} 
    \]
    by 
    \[ |a| \|G\|_{L^{p'}_x}^2 \begin{cases}
            h^{-\frac{n-1}{2}\left(1-\frac{2}{p} \right)} +h^{-(n-1) +\frac{2n}{p}+\frac{1}{k+1}\left(n-1-\frac{2(n+1)}{p} \right)}  & 2\leq p \leq \frac{2(n+1)}{n-1} \\
          h^{-(n-1) +\frac{2n}{p}+\frac{1}{k+1}\left(n-1-\frac{2(n+1)}{p} \right)} 2^{j\left(n-1-\frac{2(n+1)}{p} \right)}    & \frac{2(n+1)}{n-1} \leq p\leq \infty 
     \end{cases} 
   \]
   for $|a|\leq 1$  and 
\[\|G\|_{L^{p'}_x}^2 \begin{cases}
        |a|^{1-\frac{n-1}{2}+\frac{n+1}{p}} h^{-\frac{n-1}{2}\left(1-\frac{2}{p} \right)} + |a|h^{-(n-1) +\frac{2n}{p}+\frac{1}{k+1}\left(n-1-\frac{2(n+1)}{p} \right)} 2^{j\left(n-1-\frac{2(n+1)}{p} \right)}  & 2\leq p \leq \frac{2(n+1)}{n-1} \\ 
        |a| h^{-(n-1) +\frac{2n}{p}+\frac{1}{k+1}\left(n-1-\frac{2(n+1)}{p} \right)} 2^{j\left(n-1-\frac{2(n+1)}{p} \right)}  & \frac{2(n+1)}{n-1} \leq p\leq \infty 
     \end{cases}.
\]
    for $|a|\geq 1$.
    Lastly, taking the largest term in the case where $2\leq p \leq \frac{2(n+1)}{n-1}$ and taking the square root of both sides, we finally have 
    \[
    \left\| \int W_{a,j}(x_1)G(x_1,\xb) dx_1 \right\|_{L^2_{b,\xib}} \lesssim h^{-\delta(n,p,k)} |a|^{1/2}\|G\|_{L^{p'}_x} \begin{cases}
            1 & 2\leq p \leq \frac{2(n+1)}{n-1} \\
              2^{j\left( \frac{n-1}{2}-\frac{n+1}{p}\right)} & \frac{2(n+1)}{n-1}\leq p\leq \infty 
     \end{cases},
    \]
    for $|a|\leq 1$, and 
    \[
    \left\| \int W_{a,j}(x_1)G(x_1,\xb) dx_1 \right\|_{L^2_{b,\xib}} \lesssim h^{-\delta(n,p,k)} \|G\|_{L^{p'}_x} \begin{cases}
          |a|^{\frac{1}{2}+\frac{1}{p}-\frac{n-1}{2}\left(\frac{1}{2}-\frac{1}{p} \right)} & 2\leq p \leq \frac{2(n+1)}{n-1} \\ 
           |a|^{1/2} 2^{j\left( \frac{n-1}{2}-\frac{n+1}{p}\right)}  & \frac{2(n+1)}{n-1}\leq p\leq \infty 
     \end{cases}
    \]
    for $|a|\geq 1$.
     \end{proof}

\appendix
\section{Stationary Phase and Van der Corput}\label{sec:app}
In the proof of Lemma \ref{lem:W*term} we wanted to use the Stationary Phase Lemma to bound an oscillatory integral of the form
\begin{equation}\label{eqn:WantSP}
 \left| \int e^{\frac{i}{h} \left( \varphi_1(x_1,\xb,\xib)-\varphi_1(z_1,\bar{z},\xib) \right)} B(x,z,\xib) \psi_{j}(\xib) d\xib  \right|.
\end{equation}
As noted, the Stationary Phase Lemma cannot be directly applied as $\psi_{j}(\xib)$ is dependent on $h$. Many works deal with such oscillatory integrals, making various assumptions on the regularity of the phase, amplitude, dimension, etc.. We include a proof here for completeness. The proof follows closely to \cite[Proposition 1.1]{TacySP}.
\begin{proposition}\label{prop:VDC} Let 
\[
I(h,x)=\int e^{\frac{i}{h}\phi(x,\xi)}a_h(x,\xi)d\xi, \quad \xi\in\R^d
\]
where $\phi$ (independent of $h$) has a non-degenerate critical point at $\xi(x)$ such that
\[
|\det \partial^2_\xi \phi(x,\xi(x))|\gtrsim \mu^d \quad \text{and} \quad |\partial_\xi^\alpha \phi(x,\xi)|\lesssim_{\alpha} \mu,
\]
and $a_h$ smooth and compactly supported satisfying $|\partial_\xi^\alpha a_h(x,\xi)|\lesssim_\alpha f(h)^{|\alpha|}$ for any multiindex $\alpha$. Moreover, assume that $f(h)\leq h^{-\frac{1}{2}}\mu^{\frac{1}{2}}$. Then
\[
|I(h,x)|\lesssim h^{\frac{d}{2}}\mu^{-\frac{d}{2}}.
\]
\end{proposition}

To obtain the bound on \ref{eqn:WantSP} used in Lemma \ref{lem:W*term} we employ Proposition \ref{prop:VDC} for $\phi=\varphi_1(x_1,\xb,\xib)-\varphi_1(z_1,\bar{z},\xib)$, $a_h=B(x,z,\xib) \psi_{j}(\xib)$, $d=n-1$, and $\mu=|x_1-z_1|$. Moreover note that
\[
|\partial_{\xib}^\alpha \left(B(x,z,\xib) \psi_{j}(\xib) \right)| \lesssim \left( 2^{-j}h^{-\frac{1}{k+1}} \right)^{|\alpha|}
\]
and so we are justified using this proposition for 
\[
2^{-j}h^{-\frac{1}{k+1}}\leq h^{-\frac{1}{2}}\mu^{\frac{1}{2}} \implies 2^{-2j}h^{1-\frac{2}{k+1}}\leq|x_1-z_1|.
\]
\begin{proof}
    We begin by using Taylor's Theorem with Remainder on the phase $\phi$ around the critical point $\xi(x)$:
    \begin{align*}
    \phi(x,\xi)&=\phi(x,\xi(x))+\partial_\xi \phi(x,\xi(x))(\xi-\xi(x))+\frac{1}{2}(\xi-\xi(x))^T \operatorname{Hess}\phi\big|_{\xi(x)}(\xi-\xi(x))+R_2(x,\xi,\xi(x)) \\
    &=\phi(x,\xi(x))+\frac{1}{2}(\xi-\xi(x))^T \operatorname{Hess}\phi\big|_{\xi(x)}(\xi-\xi(x))+O(\mu|\xi-\xi(x)|^2).
    \end{align*}
    Therefore
    \[
    I(h,x)=e^{\frac{i}{h}\phi(x,\xi(x))}\int e^{\frac{i}{h}\tilde{\phi}(x,\xi)}a_h(x,\xi)d\xi =: e^{\frac{i}{h}\phi(x,\xi(x))} \tilde{I}(h,x)
    \]
    where $\tilde{\phi}(x,\xi)=\phi(x,\xi)-\phi(x,\xi(x))$. Next let $\chi\in\cci(\R)$ such that
    \[
    \chi(t)=\begin{cases} 1 & |t|\leq 1 \\ 0 & |t|\geq 2 \end{cases}.
    \]
    We decompose $\tilde{I}(h,x)$ into two integrals by using $\chi$ to cutoff the piece at scale $s_c=h^{\frac{1}{2}}\mu^{-\frac{1}{2}}$ around the critical point $\xi(x)$:
    \begin{multline*}
    \tilde{I}(h,x)=\int e^{\frac{i}{h}\tilde{\phi}(x,\xi)}\chi(h^{-\frac{1}{2}}\mu^{\frac{1}{2}}|\xi-\xi(x)|)a_h(x,\xi)d\xi+\int e^{\frac{i}{h}\tilde{\phi}(x,\xi)}\left(1-\chi(h^{-\frac{1}{2}}\mu^{\frac{1}{2}}|\xi-\xi(x)|) \right)a_h(x,\xi)d\xi \\
    =: J(h,x)+K(h,x).
      \end{multline*}
      Thus, once we bound $|J|$ and $|K|$ we will be done since
      \[
      |I(h,x)|=\left|e^{\frac{i}{h}\phi(x,\xi(x))} \tilde{I}(h,x)\right| \leq |J(h,x)|+|K(h,x)|.
      \]
      We simply control $J$ using the support properties of $\chi(h^{-\frac{1}{2}}\mu^{\frac{1}{2}}|\xi-\xi(x)|)$ to obtain
      \[
      |J(h,x)|\lesssim (h^{\frac{1}{2}}\mu^{-\frac{1}{2}})^d.
      \]
      The bound on $|K|$ is a bit more involved. First we note
      \[
      \partial_{\xi}\tilde{\phi}(x,\xi)=\frac{1}{2}\operatorname{Hess}\phi\big|_{\xi(x)}(\xi-\xi(x))+\frac{1}{2}(\xi-\xi(x))^T\operatorname{Hess}\phi\big|_{\xi(x)}+O(\mu|\xi-\xi(x)|).
      \]
      Therefore, by a partition of unity argument we can assume that there is some $j$ such that
      \begin{equation}\label{eqn:xikbdd}
      \left| \partial_{\xi_j}\tilde{\phi}\right| \gtrsim \mu |\xi-\xi(x)|.
      \end{equation}
      Without loss of generality let $j=1$. Define the operator
      \[
      \mathcal{L}=\frac{h}{i \partial_{\xi_1}\tilde{\phi}} \partial_{\xi_1}
      \]
      and note
      \[
      \mathcal{L} \left( e^{\frac{i}{h}\tilde{\phi}(x,\xi)} \right) =e^{\frac{i}{h}\tilde{\phi}(x,\xi)}.
      \]
      Therefore,
      \[
      K(h,x)=\int \mathcal{L}^N \left(e^{\frac{i}{h}\tilde{\phi}(x,\xi)} \right)\left(1-\chi(h^{-\frac{1}{2}}\mu^{\frac{1}{2}}|\xi-\xi(x)|) \right)a_h(x,\xi)d\xi 
      \]
      and integrating by parts we obtain
      \[
      K(h,x)=\int e^{\frac{i}{h}\tilde{\phi}(x,\xi)} {\mathcal{L}^*}^N \left( \left(1-\chi(h^{-\frac{1}{2}}\mu^{\frac{1}{2}}|\xi-\xi(x)|) \right)a_h(x,\xi) \right) d\xi, 
      \]
      where $\mathcal{L}^*u=\frac{-h}{i}\partial_{\xi_1}\left( \frac{1}{\partial_{\xi_1}\tilde{\phi}}u \right)$. We note that 
      \[
      \left|{\mathcal{L}^*}^N \left( \left(1-\chi(h^{-\frac{1}{2}}\mu^{\frac{1}{2}}|\xi-\xi(x)|) \right)a_h(x,\xi) \right) \right|
      \]
      is composed of a sum of terms of the form
      \[
     h^N \left( \partial_{\xi_1}^{N_1} ( \partial_{\xi_1} \tilde\phi)^{-N}  \right) \left( \partial_{\xi_1}^{N_2} \left(1-\chi(h^{-\frac{1}{2}}\mu^{\frac{1}{2}}|\xi-\xi(x)|) \right)  \right)\left( \partial_{\xi_1}^{N_3} a_h(x,\xi) \right) 
      \]
      where $N_1+N_2+N_3=N$. Using that $|\partial_\xi^\alpha \phi| \lesssim \mu$, that $\tilde{\phi} $ satisfies (\ref{eqn:xikbdd}), and Fa\`{a} di Bruno's formula we can control 
      \begin{align*}
    \left|\partial_{\xi_1}^{N_1}  ( \partial_{\xi_1} \tilde\phi)^{-N}\right|  &=  \left| \sum C_{N,N_1,m_1,\dots,m_n} (\partial_{\xi_1} \tilde\phi)^{-(N+m_1+\cdots+m_{N_1})} \prod_{j=1}^{N_1} \left( \partial_{\xi_1}^j \partial_{\xi_1}\tilde{\phi} \right)^{m_j} \right| \\
    &\lesssim \sum (\mu |\xi-\xi(x)|)^{-(N+m_1+\cdots+m_{N_1})} \mu^{m_1+\cdots m_{N_1}}
      \end{align*}
      where the sum is over all $N_1$-tuples of non-negative integers $(m_1,m_2,\dots,m_{N_1})$ such that 
      \[1m_1+2m_2+\dots + N_1 m_{N_1}=N_1.\]
      Since $|\xi-\xi(x)|\geq h^{\frac{1}{2}}\mu^{-\frac{1}{2}}$, we have
      \[
      \sum (\mu |\xi-\xi(x)|)^{-(N+m_1+\cdots+m_{N_1})} \mu^{m_1+\cdots m_{N_1}}\leq \sum (\mu h)^{-\frac{N+m_1+\cdots+m_{N_1}}{2}} \mu^{m_1+\cdots m_{N_1}},
      \]
      moreover, this helps us to see that the leading order term occurs when $m_1+m_2+\dots+m_{N_1}$ is largest, i.e.\ $m_1+m_2+\dots+m_{N_1}=N_1$. Thus we will use the estimate
      \begin{align*}
       \left|\partial_{\xi_1}^{N_1}  ( \partial_{\xi_1} \tilde\phi)^{-N}\right| &\lesssim  \sum (\mu |\xi-\xi(x)|)^{-(N+m_1+\cdots+m_{N_1})} \mu^{m_1+\cdots m_{N_1}} \\ & \lesssim (\mu |\xi-\xi(x)|)^{-N-N_1} \mu^{N_1}=\mu^{-N}(|\xi-\xi(x)|^{-N-N_1}).
      \end{align*}
      Next, we note
      \[
     \left|  \partial_{\xi_1}^{N_2} \left(1-\chi(h^{-\frac{1}{2}}\mu^{\frac{1}{2}}|\xi-\xi(x)|) \right) \right| \lesssim (h^{-\frac{1}{2}}\mu^{\frac{1}{2}})^{N_2},
      \]
      and using that$|\partial_{\xi}^\alpha a_h|\lesssim f(h)^{|\alpha|}$ where $f(h)\leq h^{-\frac{1}{2}}\mu^{\frac{1}{2}}$ we can control 
      \[
     \left|  \partial_{\xi_1}^{N_3} a_h(x,\xi) \right| \lesssim f(h)^{N_3} \leq \left(h^{-\frac{1}{2}}\mu^{\frac{1}{2}}\right)^{N_3}.
      \]
      Combining, we obtain
      \begin{align*}
      |K(h,x)| &\leq \int \left|{\mathcal{L}^*}^N \left( \left(1-\chi(h^{-\frac{1}{2}}\mu^{\frac{1}{2}}|\xi-\xi(x)|) \right)a_h(x,\xi) \right) \right| d\xi \\
      &\lesssim h^N \mu^{-N} (h^{-\frac{1}{2}}\mu^{\frac{1}{2}})^{N_2+N_3}\int_{|\xi-\xi(x)|\geq h^{\frac{1}{2}}\mu^{-\frac{1}{2}}} |\xi-\xi(x)|^{-(N+N_1)}d\xi \\
      & \lesssim  h^N \mu^{-N} (h^{-\frac{1}{2}}\mu^{\frac{1}{2}})^{N_2+N_3} \left( h^{\frac{1}{2}}\mu^{-\frac{1}{2}} \right)^{-N-N_1+d} =h^{d/2}\mu^{-d/2}
      \end{align*}
      where we used that $N=N_1+N_2+N_3.$ Finally we have
      \[
      |I(h,x)|\leq |J(h,x)|+|K(h,x)|\lesssim h^{\frac{d}{2}}\mu^{-\frac{d}{2}}
      \]
      as desired.
\end{proof}

\bibliographystyle{abbrv}
\bibliography{Quasimode}
\end{document}